\documentclass[12pt, a4paper]{article}
\usepackage{amssymb,amsmath,amsthm}
\usepackage{fancyhdr}
\usepackage[british]{babel}
\usepackage{enumitem}
\usepackage{multirow}
\usepackage{diagbox}
\usepackage{makecell}
\usepackage{tikz}
\usetikzlibrary{decorations.pathreplacing}
\usepackage{hyperref}
\usepackage[margin=2.5cm]{geometry}

\newtheorem{theorem}{Theorem}[section]
\newtheorem{proposition}[theorem]{Proposition}
\newtheorem{lemma}[theorem]{Lemma}

\newtheorem{definition}[theorem]{Definition}
\theoremstyle{definition}
\newtheorem{case}{Case}

\newtheorem{observation}{Observation}

\newtheorem{conjecture}{Conjecture}
\newtheorem{corollary}[theorem]{Corollary}

\begin{document}
\title{\textbf{Gallai-Ramsey multiplicity for rainbow small trees\footnote{Supported by NSFC No.12131013 and 12161141006.}}
}
\author{\small Xueliang Li, Yuan Si\\
\small Center for Combinatorics, Nankai University, Tianjin 300071, China\\
\small \tt lxl@nankai.edu.cn, yuan\_si@aliyun.com}
\date{}
\maketitle

\begin{abstract}
Let $G, H$ be two non-empty graphs and $k$ be a positive integer. The Gallai-Ramsey number $\operatorname{gr}_k(G:H)$ is defined as the minimum positive integer $N$ such that for all $n\geq N$, every $k$-edge-coloring of $K_n$ contains either a rainbow subgraph $G$ or a monochromatic subgraph $H$. The Gallai-Ramsey multiplicity $\operatorname{GM}_k(G:H)$ is defined as the minimum total number of rainbow subgraphs $G$ and monochromatic subgraphs $H$ for all $k$-edge-colored $K_{\operatorname{gr}_k(G:H)}$. In this paper, we get some exact values of the Gallai-Ramsey multiplicity for rainbow small trees versus general monochromatic graphs under a sufficiently large number of colors. We also study the bipartite Gallai-Ramsey multiplicity.\\[3mm]
\end{abstract}

\section{Introduction}
In this paper, the graphs we consider are finite, undirected, simple and without isolated vertices. Let $V(G)$ and $E(G)$ denote the vertex set and edge set of a graph $G$, respectively. An \emph{edge-coloring} of $G$ is a function $c:E(G)\rightarrow \{1,2,\ldots,k\}$, where $\{1,2,\ldots,k\}$ is called the set of colors. We can also use red, blue or other specific names to represent these colors. An edge-colored graph is called \emph{rainbow} if all its edges have distinct colors and \emph{monochromatic} if all its edges have the same color. A $k$-edge-coloring of a graph is \emph{exact} if all the $k$ colors are used at least once. In this paper, we only consider exact edge-colorings of graphs.

The \emph{union} $G\cup H$ of two graphs $G$ and $H$ is the graph with vertex set $V(G)\cup V(H)$ and edge set $E(G)\cup E(H)$. The \emph{degree}, $\deg_G(v)$ or $\deg(v)$ for short, of a vertex $v$ of $G$ is the number of edges incident to $v$ in $G$. We usually say that a vertex with degree $1$ is a \emph{leaf vertex}, and an edge incident to a leaf vertex is called a \emph{pendent edge}. A path with $n$ vertices from $v_1$ to $v_n$ is denoted as $P_n=v_1v_2\ldots v_n$ or $P_n=e_1e_2\ldots e_{n-1}$, which is a vertex-edge alternative sequence $v_1,e_1,v_2,e_2,\ldots,v_{n-1},e_{n-1},v_n$ such that $v_1,v_2,\ldots,v_n$ are distinct vertices, $e_1,e_2,\ldots,e_{n-1}$ are distinct edges and $v_iv_{i+1}= e_i$ for each $i\in\{1,2,\ldots,n-1\}$. $S_3^{+}$ is a the graph consisting of a triangle with one pendent edge and $P_4^{+}$ is a the graph consisting of $P_4$ with one pendent edge incident with an inner vertex of $P_4$. $P_4^{+}$ can also be seen as the graph by adding an extra pendent edge to a leaf vertex of $K_{1,3}$. For convenience, we call the newly added pendent edge at $K_{1,3}$ the \emph{tail edge} of $P_4^{+}$.

The \emph{automorphism group} of a graph $G$ is denoted as $\operatorname{Aut}(G)$. For the automorphism groups of some special graphs, we have the following conclusions: $\operatorname{Aut}(K_n)\cong S_n, \operatorname{Aut}(K_{m,n})\cong S_m\times S_n$ for $m\ne n$, and $\operatorname{Aut}(P_{n})\cong S_2$, where $S_n$ is the $n$-order symmetric group. For more notation and terminology not defined here, we refer to~\cite{BondyMurty}.

We also define some notations to replace some text descriptions in this paper. We use $e_1\sim e_2$ to denote two adjacent edges $e_1$ and $e_2$; similarly, we use $e_1\nsim e_2$ to denote two non-adjacent edges $e_1$ and $e_2$. Given a subgraph $H$ of a graph $G$, we define $\operatorname{Num}_{G}(H)$ to be the number of different copies of $H$ in $G$. Furthermore, if $e_1,e_2,\ldots,e_i$ are edges of graph $G$, we use $\operatorname{Num}_{G}(H|e_1,e_2,\ldots,e_i)$ to denote the number of different copies of $H$ that contain edges $e_1,e_2,\ldots,e_i$ in $G$. For the balanced complete $(k-1)$-part graph $K_{\underbrace{2,2,\ldots,2}_{k-1}}$, we re-write it as $K_{(k-1)\times 2}$.

\subsection{Gallai-Ramsey number and multiplicity}
In 1930, Ramsey problems were first studied by Ramsey in~\cite{Ramsey}. Given two graphs $G$ and $H$, the \emph{Ramsey number} $r(G,H)$ is defined as the minimum positive integer $n$ such that every red/blue-edge-coloring of $K_{n}$ contains either a red subgraph $G$ or a blue subgraph $H$. If $G=H$, then we simply denote $r(G, H)$ as $r(G)$. More generally, the definition of Ramsey number has been extended to multicolor and hypergraphs, and there are currently many research results available. Determining the exact value of the Ramsey numbers or improving the known upper or lower bounds on the number has always been a hot research topic in graph theory. For more results on Ramsey numbers, we refer to~\cite{Radziszowski}.

In 2010, Faudree, Gould, Jacobson and Magnant in \cite{FGJM10} provided a definition of a rainbow version of the Ramsey number, called the \emph{Gallai-Ramsey number}.
\begin{definition}{\upshape \cite{FGJM10}}\label{Def:GR}
Given two non-empty graphs $G,H$ and a positive integer $k$, define the Gallai-Ramsey number $\operatorname{gr}_k(G:H)$ to be the minimum integer $N$ such that for all $n\ge N$, every $k$-edge-coloring of $K_n$ contains either a rainbow subgraph $G$ or a monochromatic subgraph $H$.
\end{definition}

Considering the Gallai-Ramsey number on $k$-edge-colored balanced complete bipartite graph $K_{n,n}$ is another research direction. In 2019, Li, Wang and Liu in~\cite{LWL19} gave the definition of \emph{bipartite Gallai-Ramsey number}.

\begin{definition}{\upshape \cite{LWL19}}\label{Def:bGR}
Given two non-empty bipartite graphs $G, H$ and a positive integer $k$, define the bipartite Gallai-Ramsey number $\operatorname{bgr}_k(G:H)$ to be the minimum integer $N$ such that for all $n\ge N$, every $k$-edge-coloring of $K_{n,n}$ contains either a rainbow subgraph $G$ or a monochromatic subgraph $H$.
\end{definition}

In the past decade, there has been a wealth of research on the Gallai-Ramsey numbers. In terms of current research progress, six types of rainbow graphs have been studied, which are $K_3, S_3^{+}, K_{1,3}, P_4, P_5$ and $P_4^{+}$. An edge-coloring of a complete graph without rainbow triangles is called a Gallai coloring. As early as 2010, Gy\'{a}rf\'{a}s, S\'{a}rk\"{o}zy, Seb\H{o} and Selkow studied the Ramsey problem of complete graphs under Gallai coloring in~\cite{GSSS2010}, although they did not use the definition of Gallai-Ramsey formally. The Gallai-Ramsey number involving rainbow triangle has received widespread attention. One of the most important conjectures was proposed by Fox, Grinshpun and Pach in their 2015 paper~\cite{FoxGrinshpunPach}.
\begin{conjecture}{\upshape \cite{FoxGrinshpunPach}}\label{Conj}
For integers $k\ge 1$ and $n\ge 3$,
\begin{equation*}
	\operatorname{gr}_k(K_3:K_n)=\left\{
	\begin{array}{ll}
	  (r(K_n)-1)^{k/2}+1 , & \text{if $k$ is even,}\\
		(n-1)(r(K_n)-1)^{(k-1)/2}+1 , & \text{if $k$ is odd.}\\
	\end{array}
	\right
	.
\end{equation*}
\end{conjecture}
Conjecture~\ref{Conj} has been solved in some special cases. When $n=3$, Gy\'{a}rf\'{a}s, S\'{a}rk\"{o}zy, Seb\H{o} and Selkow gave a simple proof in~\cite{GSSS2010}. In~\cite{LiuMagnantSaitoSchiermeyerShi}, Liu, Magnant, Saito, Schiermeyer and Shi solved the conjecture when $n=4$. In~\cite{MagnantSchiermeyer}, Magnant and Schiermeyer studied the case of $n=5$ and concluded that only one of the conjectures of Conjecture~\ref{Conj} and $r(K_5)=43$~\cite{McKayRadziszowski} could be true, while the other was false. In~\cite{LiBroersmaWangGraphTheory}, Li, Broersma and Wang studied other extremal problems related to Gallai coloring.

For the research on Gallai-Ramsey number involving rainbow $S_3^{+}$, we refer to~\cite{FujitaMagnant,LiWangDiscrete,LiWangDMGT}. For rainbow subgraphs $K_{1,3}$ and $P_4^{+}$, we refer to~\cite{BassMagnantOzekiPyron,CJMW23}, and for rainbow subgraphs $P_4$ and $P_5$, we refer to~\cite{LiBesseMagnantWang,LWL19,ZhouLiMaoWei,ZWLM2023}. For more results about Gallai-Ramsey numbers, we refer to the monograph~\cite{MagnantNowbandegani} and the survey paper~\cite{FujitaMagnantOzeki}.

Recently, a counting problem related to the Gallai-Ramsey number has been studied. Li, Broersma and Wang in~\cite{LiBroersmaWangDiscrete} studied the minimum number of subgraphs $H$ in all $k$-edge-colored complete graphs without rainbow triangles (also known as Gallai coloring). Later, Mao in \cite{Mao} proposed the definition of \emph{Gallai-Ramsey multiplicity}.

\begin{definition}{\upshape \cite{Mao}}\label{Def:GM}
Given two non-empty graphs $G, H$ and a positive integer $k$, define the Gallai-Ramsey multiplicity $\operatorname{GM}_k(G:H)$ to be the minimum total number of rainbow subgraphs $G$ and monochromatic subgraphs $H$ for all $k$-edge-colored $K_{\operatorname{gr}_k(G:H)}$.
\end{definition}

Based on the definitions of bipartite Gallai-Ramsey number and Gallai-Ramsey multiplicity, we give the definition of \emph{bipartite Gallai Ramsey multiplicity}.
\begin{definition}\label{Def:bi-GM}
Given two non-empty bipartite graphs $G, H$ and a positive integer $k$, define the bipartite Gallai-Ramsey multiplicity $\operatorname{bi-GM}_k(G:H)$ to be the minimum total number of rainbow subgraphs $G$ and monochromatic subgraphs $H$ for all $k$-edge-colored $K_{\operatorname{bgr}_k(G:H),\operatorname{bgr}_k(G:H)}$.
\end{definition}

\subsection{Structural theorems under rainbow-tree-free colorings}
The five $k$-edge-colored structures of complete graphs or complete bipartite graphs given below is for convenience of describing several structural theorems.

\noindent \textbf{Colored Structure 1:} Let $(V_1,V_2,\ldots,V_k)$ be a partition of $V(K_n)$ such that for each $i$, all the edges connecting two vertices in $V_i$ are colored by either $1$ or $i$ and all the edges between $V_i$ and $V_j$ with $i\ne j$ are colored by $1$.

\noindent \textbf{Colored Structure 2:} Let $K_n$ be a $k$-edge-colored complete graph such that $K_n-v$ is monochromatic for some vertex $v$.

\noindent \textbf{Colored Structure 3:} Let $(U,V)$ be the bipartition of complete bipartite graph $K_{n,n}$. $U$ can be partitioned into $k$ non-empty parts $U_1, U_2,\ldots,U_k$ such that all the edges
between $U_i$ and $V$ have color $i$ for $i\in\{1,2,\ldots,k\}$.

\noindent \textbf{Colored Structure 4:} Let $(U,V)$ be the bipartition of complete bipartite graph $K_{n,n}$. $U$ can be partitioned into two parts $U_1$ and $U_2$ with $|U_1|\ge 1, |U_2|\ge 0$, and $V$ can be partitioned into $k$ parts $V_1,V_2,\ldots,V_k$ with $|V_1|\ge 0$ and $|V_j|\ge 1$, $j\in\{2,3,\ldots,k\}$, such that all the edges between $V_i$ and $U_1$ have color $i$ and all the edges between $V_i$ and $U_2$ have color $1$ for $i\in\{1,2,\ldots,k\}$.

\noindent \textbf{Colored Structure 5:} Let $(U,V)$ be the bipartition of complete bipartite graph $K_{n,n}$. $U$ can be partitioned into $k$ parts $U_1,U_2,\ldots,U_k$ with $|U_1|\ge 0, |U_j|\ge 1$ and $V$ can be partitioned into $k$ parts $V_1,V_2,\ldots,V_k$ with $|V_1|\ge 0, |V_j|\ge 1$, $j\in\{2,3,\ldots,k\}$, such that only colors $1$ and $i$ can be used on the edges between $U_i$ and $V_i$ for $i\in\{1,2,\ldots,k\}$, and all the other edges have color $1$.

Thomason and Wagner in~\cite{ThomasonWagner2007} obtained the following results.

\begin{theorem}{\upshape \cite{ThomasonWagner2007}}\label{th-3-path-Structure}
For an integer $n\geq 4$, let $K_n$ be an edge-colored complete graph with at least three colors so that it contains no rainbow $P_4$ if and only if $n=4$ and three colors are used, each color forming a perfect matching.
\end{theorem}

\begin{theorem}{\upshape \cite{ThomasonWagner2007}}\label{th-4-path-Structure}
For integers $k\ge 5$ and $n\ge 5$, let $K_n$ be a $k$-edge-colored complete graph so that it contains no rainbow $P_5$ if and only if Colored Structure 1 or Colored Structure 2 occurs.
\end{theorem}

Bass, Magnant, Ozeki and Pyron in~\cite{BassMagnantOzekiPyron} obtained the following results. The study of this structural theorem can be traced back to the study of local $k$-coloring Ramsey numbers by Gy{\'a}rf{\'a}s, Lehel, Schelp and Tuza in~\cite{GLST1987}.

\begin{theorem}{\upshape \cite{BassMagnantOzekiPyron,GLST1987}}\label{th-Star-Structure}
For integers $k\ge 4$ and $n\ge 4$, let $K_n$ be a $k$-edge-colored complete graph so that it contains no rainbow $K_{1,3}$  if and only if Colored Structure 1 occurs.
\end{theorem}

Bass, Magnant, Ozeki and Pyron in~\cite{BassMagnantOzekiPyron} described the colored structure of complete graphs without rainbow $P_{4}^{+}$. Also, Schlage-Puchta and Wagner proved this structural theorem in~\cite{Schlage-PuchtaWagner}.

\begin{theorem}{\upshape \cite{BassMagnantOzekiPyron,Schlage-PuchtaWagner}}\label{th-P4+-Structure}
For integers $k\ge 5$ and $n\ge 5$, let $K_n$ be a $k$-edge-colored complete graph so that it contains no rainbow $P_{4}^{+}$ if and only if Colored Structure 1 occurs.
\end{theorem}

In terms of edge-colorings of complete bipartite graphs, Li, Wang and Liu in~\cite{LWL19} first obtained the following results.

\begin{theorem}{\upshape \cite{LWL19}}\label{th-P4-bGR-Structure}
Let $(U,V)$ be the bipartition of a complete bipartite graph. For integers $k\ge 3$ and $n\ge 2$, let $K_{n,n}$ be a $k$-edge-colored complete bipartite graph so that it contains no rainbow $P_4$ if and only if Colored Structure 3 occurs.
\end{theorem}

\begin{theorem}{\upshape \cite{LWL19}}\label{th-P5-bGR-Structure}
Let $(U,V)$ be the bipartition of a complete bipartite graph. For integers $k\ge 5$ and $n\ge 3$, let $K_{n,n}$ be a $k$-edge-colored complete bipartite graph so that it contains no rainbow $P_5$ if and only if Colored Structure 4 or Colored Structure 5 occurs.
\end{theorem}

Li and Wang in~\cite{LiWangDAM} first described the colored structure of complete bipartite graphs without rainbow $K_{1,3}$. Recently, Chen, Ji, Mao and Wei in~\cite{CJMW23} repeoved this structural theorem on balanced complete bipartite graphs.

\begin{theorem}{\upshape \cite{CJMW23,LiWangDAM}}\label{th-K13-bGR-Structure}
Let $(U,V)$ be the bipartition of a complete bipartite graph. For integers $k\ge 5$ and $n\ge 3$, let $K_{n,n}$ be a $k$-edge-colored complete bipartite graph so that it contains no rainbow $K_{1,3}$ if and only if Colored Structure 5 occurs.
\end{theorem}

It should be noted that the structural theorems cited above are only a part of what is needed in this paper. For a complete survey of the structural theorems, we refer to the original references.

\subsection{Main results}
Due to the fact that the research in this paper is based on the exact $k$-edge-colorings, it is natural to require that the number of edges in the graph is at least the number of colors that we consider. From the above structural theorems, it can be seen that if there are no rainbow subgraphs $G\in\{P_4,P_5,K_{1,3},P_4^{+}\}$ for $k$-edge-colored complete graphs or complete bipartite graphs, then there must be some monochromatic graphs under this colored structure. For example, a $k$-edge-colored complete graph without rainbow $K_{1,3}$ must contain a monochromatic $K_{(k-1)\times 2}$; a $k$-edge-colored balanced complete bipartite graph without rainbow $P_4$ must contain a monochromatic $K_{1,k}$; a $k$-edge-colored balanced complete bipartite graph without rainbow $P_5$ or $K_{1,3}$ must contain a monochromatic $K_{1,\left\lceil\frac{k-1}{2}\right\rceil}$. Based on these properties, we know that when the number of colors $k$ is sufficiently large with respect to the subgraph $H$, the Gallai-Ramsey number $\operatorname{gr}_k(G:H)$ (also bipartite Gallai-Ramsey number $\operatorname{bgr}_k(G:H)$) does not depend on the subgraph $H$, but only on $k$.

In the third section of this paper, we give some exact values of the Gallai-Ramsey number and multiplicity when the number of colors $k=\binom{t}{2}$ is sufficiently large with respect to the subgraph $H$. The main results for the Gallai-Ramsey multiplicity are shown in the following table.
\begin{table}
	\begin{center}
		\begin{tabular}{|c|c|c|}
			\hline
			& $\operatorname{GM}_{k-1}(G:H)$ & $\operatorname{GM}_{k-2}(G:H)$ \\[0.1cm]
			\hline
			$G=K_{1,3}$ & \makecell[c]{$3\ (t=4)$\\[0.1cm] $18 \ (t=5)$\\[0.1cm] $(t-1)\binom{t-1}{3}+\binom{t-3}{3}+2\binom{t-3}{2} \ (t\ge 6)$} & \makecell[c]{$1\ (t=4)$ \\[0.1cm] $14\ (t=5)$ \\[0.1cm] $(t-3)\binom{t-1}{3}+3\binom{t-3}{3}+6\binom{t-3}{2} \ (t\ge 6)$} \\[0.1cm]
			\hline
			$G=P_{4}^{+}$ & \makecell[c]{$60\binom{t}{5}-5(t-3)(t-4) \ (t\ge 5)$} & \makecell[c]{$60\binom{t}{5}-15(t-3)(t-4) \ (t\ge 5)$} \\[0.1cm]
			\hline
			$G=P_4$ & \makecell[c]{$8 \ (t=4)$\\[0.1cm] $12\binom{t}{4}-2(t-3) \ (t\ge 5)$  } & \makecell[c]{$4 \ (t=4)$\\[0.1cm] $12\binom{t}{4}-6(t-3) \ (t\ge 5)$ }  \\[0.1cm]
			\hline
			$G=P_5$ & \makecell[c]{$60\binom{t}{5}-12(t-4)\ (5\le t\le 6)$ \\[0.1cm] $60\binom{t}{5}-3(t-3)(t-4) \ (t\ge 7)$} & \makecell[c]{$38\ (t=5)$ \\[0.1cm] $288\ (t=6)$ \\[0.1cm] $60\binom{t}{5}-9(t-3)(t-4) \ (t\ge 7)$} \\[0.1cm]
			\hline
		\end{tabular}
		\caption{Main results for Gallai-Ramsey multiplicity. The graph $H$ is a subgraph of some graphs related to $k$, see Theorems~\ref{THM-GM-k-1-color-star}, \ref{THM-GM-k-2-color-star}, \ref{THM-GM-k-1-color-P4+}, \ref{THM-GM-k-2-color-P4+}, \ref{THM-GM-k-1-color-P4}, \ref{THM-GM-k-2-color-P4}, \ref{THM-GM-k-1-color-P5} and~\ref{THM-GM-k-2-color-P5} for details.}
	\end{center}
\end{table}

In the fourth section of this paper, similarly, we give some exact values of the bipartite Gallai-Ramsey number and multiplicity when the number of colors $k=t^2$ is sufficiently large with respect to the subgraph $H$. The main results for the bipartite Gallai-Ramsey multiplicity are shown in the following tables.
\begin{table}
\begin{center}
	\begin{tabular}{|c|c|c|}
		\hline
		& $\operatorname{bi-GM}_{k-1}(G:H)$ & $\operatorname{bi-GM}_{k-2}(G:H)$ \\[0.1cm]
		\hline
		$G=K_{1,3}$ & \makecell[c]{$5\ (t=3)$\\[0.1cm] $30 \ (t=4)$\\[0.1cm] $(2t-1)\binom{t}{3}+\binom{t-2}{3}+2\binom{t-2}{2} \ (t\ge 5)$} & \makecell[c]{$4\ (t=3)$ \\[0.1cm] $28\ (t=4)$\\[0.1cm] $93\ (t=5)$ \\[0.1cm] $(2t-1)\binom{t}{3}+\binom{t-3}{3}+3\binom{t-3}{2} \ (t\ge 6)$   } \\[0.1cm]
		\hline
		$G=P_4$ & \makecell[c]{$t^2(t-1)^2-2(t-1) \ (t\ge 2)$ } & \makecell[c]{$t^2(t-1)^2-6(t-1) \ (t\ge 2)$}  \\[0.1cm]
		\hline
		$G=P_5$ & \makecell[c]{$t^2(t-1)^2(t-2)-3(t-1)(t-2) \ (t\ge 3)$ }& \makecell[c]{$t^2(t-1)^2(t-2)-9(t-1)(t-2) \ (t\ge 3)$} \\[0.1cm]
		\hline
	\end{tabular}
	\caption{Main results for bipartite Gallai-Ramsey multiplicity. The graph $H$ is a subgraph of some graphs related to $k$, see Theorems~\ref{THM-biGM-k-1-color-P4}, \ref{THM-biGM-k-2-color-P4}, \ref{THM-biGM-k-1-color-P5}, \ref{THM-biGM-k-2-color-P5}, \ref{THM-biGM-k-1-color-K13} and~\ref{THM-biGM-k-2-color-K13} for details.}
\end{center}
\end{table}
\section{Preliminaries}

Some propositions and lemmas presented in section are very helpful for the proof in the third and fourth sections of this paper. In 2008, Fox in~\cite{Fox} provided a result on the total number of different subgraphs $G$ in $K_n$.

\begin{proposition}{\upshape \cite{Fox}}\label{Prop-counting}
If $G$ is a subgraph of $K_n$, then $\operatorname{Num}_{K_n}(G)=\frac{|V(G)|!\binom{n}{|V(G)|}}{|\operatorname{Aut}(G)|}$.
\end{proposition}

Li, Wang and Liu in~\cite{LWL19} determined the sharp bound of $k$ such that any $k$-edge-colored $K_n$ always has a rainbow subgraph $P_5$. Bass, Magnant, Ozeki and Pyron in~\cite{BassMagnantOzekiPyron} obtained the sharp bound of $k$ such that any $k$-edge-colored $K_n$ always has a rainbow subgraph $K_{1,3}$ by studying anti-Ramsey numbers.

\begin{proposition}{\upshape \cite{LWL19}}\label{k-color-have-rainbowP5}
For integers $n\ge 5$ and $k$ with $n+1 \le k\le \binom{n}{2}$, there is always a rainbow subgraph $P_5$ under any $k$-edge-colored $K_n$.
\end{proposition}

\begin{proposition}{\upshape \cite{BassMagnantOzekiPyron}}\label{k-color-have-rainbowK13}
For integers $n\ge 4$ and $k$ with $\lceil \frac{n+3}{2}\rceil \le k\le \binom{n}{2}$, there is always a rainbow subgraph $K_{1,3}$ under any $k$-edge-colored $K_n$.
\end{proposition}

Also, Li, Wang and Liu in~\cite{LWL19} determined the sharp bound of $k$ such that any $k$-edge-coloring of $K_{n,n}$ always has a rainbow subgraph $P_4$ or $P_5$.

\begin{proposition}{\upshape \cite{LWL19}}\label{k-color-have-rainbowP4-bipartite}
For integers $n\ge 2$ and $k$ with $n+1 \le k\le n^2$, there is always a rainbow subgraph $P_4$ under any $k$-edge-colored $K_{n,n}$.
\end{proposition}
\begin{proposition}{\upshape \cite{LWL19}}\label{k-color-have-rainbowP5-bipartite}
For integers $n\ge 3$ and $k$ with $n+2 \le k\le n^2$, there is always a rainbow subgraph $P_5$ under any $k$-edge-colored $K_{n,n}$.
\end{proposition}

Similarly, we can directly obtain the following result through Theorem~\ref{th-K13-bGR-Structure}.

\begin{proposition}\label{k-color-have-rainbowK13-bipartite}
For integers $n\ge 3$ and $k$ with $n+2 \le k\le n^2$, there is always a rainbow subgraph $K_{1,3}$ under any $k$-edge-colored $K_{n,n}$.
\end{proposition}

The following lemmas are very useful in the proofs of the third section.

\begin{lemma}\label{Lem: Conut P4 in complete graph}
Let $t\ge 4$ be an integer and $e_1,e_2$ be two edges in $K_t$. Then
\begin{equation*}
	\operatorname{Num}_{K_t}(P_4|e_1,e_2)=\left\{
	\begin{array}{ll}
		4, & \text{if $e_1\nsim e_2$};\\
		2(t-3), &  \text{if $e_1\sim e_2$}.
	\end{array}\right.
\end{equation*}
\end{lemma}
\begin{proof}
Assume that $e_1\nsim e_2$, and let $e_1=v_1v_2$, $e_2=v_3v_4$.  Since $P_4$ is a connected graph, there is an edge that connects $e_1$ and $e_2$. In this case, there are four different $P_4$, which are $v_1v_2v_3v_4$, $v_1v_2v_4v_3$, $v_2v_1v_3v_4$ and $v_2v_1v_4v_3$.

Assume that $e_1\sim e_2$. From the structure of $P_4$, it can be seen that $e_1$ and $e_2$ cannot be both pendent edges of $P_4$, and one of $e_1$ and $e_2$ must be the pendent edge of $P_4$. If $e_1$ is the pendent edge of $P_4$, then there are $t-3$ different $P_4$. By symmetry, if $e_2$ is the pendent edge of $P_4$, then there are $t-3$ different $P_4$. Therefore, in this case, there are $2(t-3)$ different $P_4$.
\end{proof}
	
\begin{lemma}\label{Lem: Conut P5 in complete graph}
Let $t\ge 5$ be an integer and $e_1,e_2$ be two edges in $K_t$. Then
\begin{equation*}
	\operatorname{Num}_{K_t}(P_5|e_1,e_2)=\left\{
	\begin{array}{ll}
		12(t-4), & \text{if $e_1\nsim e_2$};\\
		3(t-3)(t-4), &  \text{if $e_1\sim e_2$}.
	\end{array}\right.
\end{equation*}
\end{lemma}
\begin{proof}
Assume that $e_1\nsim e_2$. From the structure of $P_5$, it can be seen that one of $e_1$ and $e_2$ must be the pendent edge of $P_5$. If $e_1$ is the pendent edge of $P_5$ but $e_2$ is not the pendent edge of $P_5$, then there are $4(t-4)$ different $P_5$. By symmetry, if $e_2$ is the pendent edge of $P_5$ but $e_1$ is not the pendent edge of $P_5$, then there are $4(t-4)$ different $P_5$. If the $e_1$ and $e_2$ are both pendent edges of $P_5$, then there are $4(t-4)$ different $P_5$. Therefore, in this case, there are $12(t-4)$ different $P_5$.
	
Assume that $e_1\sim e_2$. From the structure of $P_5$, it can be seen that $e_1$ and $e_2$ cannot be both pendent edges of $P_5$. If $e_1$ is the pendent edge of $P_5$ but $e_2$ is not the pendent edge of $P_5$, then there are $(t-3)(t-4)$ different $P_5$. By symmetry, if $e_2$ is the pendent edge of $P_5$ but $e_1$ is not the pendent edge of $P_5$, then there are $(t-3)(t-4)$ different $P_5$. If neither $e_1$ nor $e_2$ are the pendent edges of $P_5$, then there are $(t-3)(t-4)$ different $P_5$. Therefore, in this case, there are $3(t-3)(t-4)$ different $P_5$.
\end{proof}

\begin{lemma}\label{Lem: Conut P4+ in complete graph}
Let $t\ge 5$ be an integer and $e_1,e_2$ be two edges in $K_t$. Then
\begin{equation*}
	\operatorname{Num}_{K_t}(P_4^{+}|e_1,e_2)=\left\{
	\begin{array}{ll}
		8(t-4), & \text{if $e_1\nsim e_2$};\\
		5(t-3)(t-4), &  \text{if $e_1\sim e_2$}.
	\end{array}\right.
\end{equation*}
\end{lemma}
\begin{proof}
Assume that $e_1\nsim e_2$, and let $e_1=v_1v_2$, $e_2=v_3v_4$. Consider $P_{4}^{+}$ with edges $e_1$ and $e_2$. From the proof of Lemma~\ref{Lem: Conut P4 in complete graph}, we know that there are four different $P_4$, which are $v_1v_2v_3v_4$, $v_1v_2v_4v_3$, $v_2v_1v_3v_4$ and $v_2v_1v_4v_3$.	 Since $P_{4}^{+}$ is a graph by adding an extra pendent edge to an inner vertex of $P_4$, it follows that there are $4\cdot 2(t-4)=8(t-4)$ different $P_{4}^{+}$.
	
Assume that $e_1\sim e_2$. If $e_1$ is the tail edge of $P_{4}^{+}$, then there are $(t-3)(t-4)$ different $P_{4}^{+}$. By symmetry, if $e_2$ is the tail edge of $P_{4}^{+}$, then there are $(t-3)(t-4)$ different $P_{4}^{+}$. If neither $e_1$ nor $e_2$ are the tail edges of $P_{4}^{+}$, then there are $3(t-3)(t-4)$ different $P_{4}^{+}$. Therefore, in this case, there are $5(t-3)(t-4)$ different $P_{4}^{+}$.
\end{proof}

Also, the following lemmas are very useful in the proofs of the fourth section.

\begin{lemma}\label{Lem: Conut P4 in complete bipartite graph}
Let $t\ge 3$ be an integer and $e_1,e_2$ be two edges in $K_{t,t}$. Then
\begin{equation*}
	\operatorname{Num}_{K_{t,t}}(P_4|e_1,e_2)=\left\{
	\begin{array}{ll}
		2, & \text{if $e_1\nsim e_2$};\\
		2(t-1), &  \text{if $e_1\sim e_2$}.
	\end{array}\right.
\end{equation*}
\end{lemma}
\begin{proof}
Assume that $e_1\nsim e_2$. Let $(X,Y)$ be the bipartition of $K_{t,t}$, $e_1=v_1v_2$, $e_2=v_3v_4$ and $v_1,v_3\in X$, $v_2,v_4\in Y$. Since $P_4$ is a connected graph, there is an edge that connects $e_1$ and $e_2$. In this case, there are two different $P_4$, which are $v_1v_2v_3v_4$ and $v_2v_1v_4v_3$.
	
Assume that $e_1\sim e_2$. From the structure of $P_4$, it can be seen that $e_1$ and $e_2$ cannot be both pendent edges of $P_4$, and one of $e_1$ and $e_2$ must be the pendent edge of $P_4$. If $e_1$ is the pendent edge of $P_4$, then there are $t-1$ different $P_4$. By symmetry, if $e_2$ is the pendent edge of $P_4$, then there are $t-1$ different $P_4$. Therefore, in this case, there are $2(t-1)$ different $P_4$.
\end{proof}

\begin{lemma}\label{Lem: Conut P5 in complete bipartite graph}
Let $t\ge 3$ be an integer and $e_1,e_2$ be two edges in $K_{t,t}$. Then
\begin{equation*}
	\operatorname{Num}_{K_{t,t}}(P_5|e_1,e_2)=\left\{
	\begin{array}{ll}
		6(t-2), & \text{if $e_1\nsim e_2$};\\
		3(t-1)(t-2), &  \text{if $e_1\sim e_2$}.
	\end{array}\right.
\end{equation*}
\end{lemma}
\begin{proof}
Assume that $e_1\nsim e_2$. From the structure of $P_5$, it can be seen that one of $e_1$ and $e_2$ must be the pendent edge of $P_5$. If $e_1$ is the pendent edge of $P_5$ but $e_2$ is not the pendent edge of $P_5$, then there are $2(t-2)$ different $P_5$. By symmetry, if $e_2$ is the pendent edge of $P_5$ but $e_1$ is not the pendent edge of $P_5$, then there are $2(t-2)$ different $P_5$. If $e_1$ and $e_2$ are both pendent edges of $P_5$, then there are $2(t-2)$ different $P_5$. Therefore, in this case, there are $6(t-2)$ different $P_5$.
	
Assume that $e_1\sim e_2$. From the structure of $P_5$, it can be seen that $e_1$ and $e_2$ cannot be both pendent edges of $P_5$. If $e_1$ is the pendent edge of $P_5$ but $e_2$ is not the pendent edge of $P_5$, then there are $(t-1)(t-2)$ different $P_5$. By symmetry, if $e_2$ is the pendent edge of $P_5$ but $e_1$ is not the pendent edge of $P_5$, then there are $(t-1)(t-2)$ different $P_5$. If neither $e_1$ nor $e_2$ are the pendent edges of $P_5$, then there are $(t-1)(t-2)$ different $P_5$. Therefore, in this case, there are $3(t-1)(t-2)$ different $P_5$.
\end{proof}

Recall that the $k$-edge-colorings studied in this paper are exact, meaning that each color is used at least once. Based on this, a basic principle is that the number of colors does not exceed the total number of edges in an edge-colored graph. So by solving the equations
\begin{equation*}
	k\le \binom{n}{2}=|E(K_n)| \mbox{ and } k\le n^2=|E(K_{n,n})|,
\end{equation*}
we obtain $n\ge \frac{1+\sqrt{1+8k}}{2}$ and $n\ge  \sqrt{k}$, respectively.
Therefore, we directly get the following basic lower bound lemma.
\begin{lemma}\label{basic-lower-bound-lemma}
	For integer $k\ge 4$, $G\in\{P_4,K_{1,3}\}$ and any graph $H$, we have
	\begin{equation*}
		\operatorname{gr}_k(G:H)\ge \left\lceil \frac{1+\sqrt{1+8k}}{2} \right\rceil.
	\end{equation*}
	For integer $k\ge 5$, $G\in\{P_4,P_5,K_{1,3}\}$ and any bipartite graph $H$, we have
	\begin{equation*}
		\operatorname{bgr}_k(G:H)\ge \left\lceil \sqrt{k}\right\rceil,
	\end{equation*}
	in particular, this lower bound also holds when $3\le k\le 4$ and $G=P_4$.
\end{lemma}
It is worth noting that the idea of Lemma~\ref{basic-lower-bound-lemma} originated from the paper of Zou, Wang, Lai and Mao in~\cite{ZWLM2023}.

\section{Results for Gallai-Ramsey multiplicity}
We consider four kinds of rainbow graphs $K_{1,3}$, $P_{4}^{+}$, $P_{4}$, $P_{5}$, respectively, in the following four subsections.

\subsection{Rainbow $K_{1,3}$}
\begin{theorem}\label{k-color-Star-H-general}
Let integer $k\ge 4$. If $H$ is a subgraph of the balanced complete $(k-1)$-partite graph $K_{(k-1)\times 2}$. Then
	\begin{equation*}
	\operatorname{gr}_k(K_{1,3}:H)=\left\lceil \frac{1+\sqrt{1+8k}}{2} \right\rceil.
	\end{equation*}
\end{theorem}
\begin{proof}
The lower bound follows from Lemma~\ref{basic-lower-bound-lemma}. Let $N_k=\left\lceil \frac{1+\sqrt{1+8k}}{2} \right\rceil$. For the upper bound, we consider an arbitrary $k$-edge-coloring of $K_N \ (N\ge N_k)$. Noticing that $N_k<2k-2$ for $k\geq 4$. If $N_k\le N\le 2k-3$, then it follows from Proposition~\ref{k-color-have-rainbowK13} that there is always a rainbow $K_{1,3}$, the result thus follows. Next we assume $N\ge 2k-2$. Suppose to the contrary that $K_N$ contains neither a rainbow subgraph $K_{1,3}$ nor a monochromatic subgraph $H$. It follows from Theorem~\ref{th-Star-Structure} that the Colored Structure 1 occurs. From exact $k$-edge-coloring, we have $|V_i|\ge 2$ for each $i \in\{2,3,\ldots,k\}$. Since $H$ is a subgraph of the balanced complete $(k-1)$-partite graph $K_{(k-1)\times 2}$, it follows that there is a monochromatic $H$, a contradiction. The result thus follows.
\end{proof}

According to Theorem~\ref{k-color-Star-H-general} and Proposition~\ref{Prop-counting}, the following corollary can be directly deduced.

\begin{corollary}\label{Cor-GM-k-color-star}
For integers $k$ and $t$ satisfying $k=\binom{t}{2}\ge 6$,  and a subgraph $H$ of the balanced complete $(k-1)$-partite graph $K_{(k-1)\times 2}$ with $|E(H)|\ge 2$, 
we have
\begin{equation*}
	\operatorname{GM}_{k}(K_{1,3}:H)=\frac{4!\binom{t}{4}}{|\operatorname{Aut}(K_{1,3})|}=t\binom{t-1}{3}.
\end{equation*}
\end{corollary}

\begin{theorem}\label{THM-GM-k-1-color-star}
For integers $k$ and $t$ satisfying $k=\binom{t}{2}\ge 6$, and a subgraph $H$ of the balanced complete $(k-2)$-partite graph $K_{(k-2)\times 2}$ with $|E(H)|\ge 3$, 
we have
\begin{equation*}
\operatorname{GM}_{k-1}(K_{1,3}:H)=\left\{
\begin{array}{ll}
	3, & t=4;\\
	18, &  t=5;\\
	(t-1)\binom{t-1}{3}+\binom{t-3}{3}+2\binom{t-3}{2}, & t\ge 6.
\end{array}\right.
\end{equation*}
\end{theorem}
\begin{proof}
It follows from Theorem~\ref{k-color-Star-H-general} that $\operatorname{gr}_{k-1}(K_{1,3}:H)=t$. Consider any $(k-1)$-edge-coloring of $K_t$. Since $|E(K_t)|=\binom{t}{2}$ and each color is used at least once, it follows that there are only two edges with the same color in $K_t$. Without loss of generality, we assume that there are two red edges $e_1$ and $e_2$. Since $|E(H)|\ge 3$, it follows that we do not need to consider the number of monochromatic $H$ in $K_t$. If $e_1\nsim e_2$, then this case is equivalent to Corollary~\ref{Cor-GM-k-color-star}. Therefore,
\begin{equation*}
	\operatorname{GM}_{k-1}(K_{1,3}:H)\le t\binom{t-1}{3}.
\end{equation*}

If $e_1\sim e_2$, then $e_1$ and $e_2$ form a red $P_3$. Let vertex $v$ be incident to the edges $e_1$ and $e_2$. We first investigate the number of rainbow copies of $K_{1,3}$ with center $v$ for $t\ge 6$. Noticing that $\deg(v)=t-1$, the number of rainbow copies of $K_{1,3}$ with center $v$ and without red edges is $\binom{t-3}{3}$, and number of rainbow copies of $K_{1,3}$ with center $v$ and with a red edge is $2\binom{t-3}{2}$. In $K_t$, there are $(t-1)\binom{t-1}{3}$ rainbow copies of $K_{1,3}$ with center in $V(K_t)\setminus \{v\}$. Therefore,
\begin{equation*}
	\operatorname{GM}_{k-1}(K_{1,3}:H)\le (t-1)\binom{t-1}{3}+\binom{t-3}{3}+2\binom{t-3}{2}.
\end{equation*}

It is easy to verify that when $t\ge 6$,
\begin{equation*}
	\min\left\{t\binom{t-1}{3},(t-1)\binom{t-1}{3}+\binom{t-3}{3}+2\binom{t-3}{2}\right\}=(t-1)\binom{t-1}{3}+\binom{t-3}{3}+2\binom{t-3}{2}.
\end{equation*} 

When $t=4$, there is no rainbow $K_{1,3}$ with center $v$. In $K_4$, each other vertex has one rainbow $K_{1,3}$. So there are three rainbow copies of $K_{1,3}$ in $K_4$. Since $3<4\binom{4-1}{3}$, it follows that $\operatorname{GM}_{5}(K_{1,3}:H)=3$.

When $t=5$, the number of rainbow copies of $K_{1,3}$ with center $v$ and without red edges is $0$, and number of rainbow copies of $K_{1,3}$ with center $v$ and with a red edge is $2$. In $K_5$, there are $4\binom{4}{3}=16$ rainbow copies of $K_{1,3}$ with center in $V(K_5)\setminus \{v\}$. Therefore, $\operatorname{GM}_{9}(K_{1,3}:H)\le 18$. Since $18<5\binom{5-1}{3}$. It follows that $\operatorname{GM}_{9}(K_{1,3}:H)=18$.
\end{proof}

\begin{theorem}\label{THM-GM-k-2-color-star}
	For integers $k$ and $t$ satisfying $k=\binom{t}{2}\ge 6$, and a subgraph $H$ of the balanced complete $(k-3)$-partite graph $K_{(k-3)\times 2}$ with $|E(H)|\ge 4$, 
we have
\begin{equation*}
	\operatorname{GM}_{k-2}(K_{1,3}:H)=\left\{
	\begin{array}{ll}
		1, & t=4;\\
		14, &  t=5;\\
		(t-3)\binom{t-1}{3}+3\binom{t-3}{3}+6\binom{t-3}{2}, & t\ge 6.\\
	\end{array}\right.
\end{equation*}
\end{theorem}
\begin{proof}
It follows from Theorem~\ref{k-color-Star-H-general} that $\operatorname{gr}_{k-2}(K_{1,3}:H)=t$. Consider a $(k-2)$-edge-coloring of $K_t$. Since $|E(H)|\ge 4$, it follows that we do not need to consider the number of monochromatic $H$ in $K_t$. Noticing that each color needs to be used at least once. We first color any $k-2$ edges in $K_t$ with $k-2$ colors, and the remaining two edges are temporarily not colored, denoted as $e_1$ and $e_2$. Next, we discuss the edges $e_1$ and $e_2$ in two cases.
\setcounter{case}{0}
\begin{case}
	The edges $e_1$ and $e_2$ have the same color.
\end{case}	
Without loss of generality, we assume that these two edges are red. According to the structure of $K_t$, it is easy to calculate that if the red edges form a $3P_2$, then there are
\begin{equation*}
f_1(t)=t\binom{t-1}{3}
\end{equation*}
rainbow copies of $K_{1,3}$ in $K_t$;
if the red edges form a $P_3\cup P_2$, then there are
\begin{equation*}
f_2(t)=(t-1)\binom{t-1}{3}+\binom{t-3}{3}+2\binom{t-3}{2}
\end{equation*}
rainbow copies of $K_{1,3}$ in $K_t$;
if the red edges form a $P_4$, then there are
\begin{equation*}
f_3(t)=(t-2)\binom{t-1}{3}+2\binom{t-3}{3}+4\binom{t-3}{2}
\end{equation*}
rainbow copies of $K_{1,3}$ in $K_t$;
if the red edges form a $K_3$, then there are
\begin{equation*}
	f_4(t)=(t-3)\binom{t-1}{3}+3\binom{t-3}{3}+6\binom{t-3}{2}
\end{equation*}
rainbow copies of $K_{1,3}$ in $K_t$; if the red edges form a $K_{1,3}$, then there are
\begin{equation*}
	f_5(t)=(t-1)\binom{t-1}{3}+\binom{t-4}{3}+3\binom{t-4}{2}
\end{equation*}
rainbow copies of $K_{1,3}$ in $K_t$.

\begin{case}
	The edges $e_1$ and $e_2$ have different colors.
\end{case}	
When the edges $e_1$ and $e_2$ form a $P_3$ in $K_t$, without loss of generality, we assume that $e_1$ is red and $e_2$ is blue. Let $V(P_3)=\{u,v,w\}$ and vertex $v$ is incident to edges $e_1$ and $e_2$. According to the structure of $K_t$, it is easy to calculate that if the other red edge is not incident to vertex $u$ or $v$, and the other blue edge is not incident to vertex $v$ or $w$, then there are
\begin{equation*}
	f_1(t)=t\binom{t-1}{3}
\end{equation*}
rainbow copies of $K_{1,3}$ in $K_t$; if the other red edge is incident to vertex $u$ or $v$, and the other blue edge is not incident to vertex $v$ or $w$, then there are
\begin{equation*}
	f_2(t)=(t-1)\binom{t-1}{3}+\binom{t-3}{3}+2\binom{t-3}{2}
\end{equation*}
rainbow copies of $K_{1,3}$ in $K_t$;
if the other red edge is incident to vertex $u$, and the other blue edge is incident to vertex $w$, then there are
\begin{equation*}
	f_3(t)=(t-2)\binom{t-1}{3}+2\binom{t-3}{3}+4\binom{t-3}{2}
\end{equation*}
rainbow copies of $K_{1,3}$ in $K_t$;
if the other red edge is incident to vertex $v$, and the other blue edge is also incident to vertex $v$, then there are
\begin{equation*}
	f_6(t)=(t-1)\binom{t-1}{3}+\binom{t-5}{3}+4\binom{t-5}{2}+4(t-5)
\end{equation*}
rainbow copies of $K_{1,3}$ in $K_t$.

When the edges $e_1$ and $e_2$ form a $2P_2$ in $K_t$, without loss of generality, we assume that $e_1$ is red and $e_2$ is blue. According to the structure of $K_t$, it is easy to calculate that if the other red edge is not adjacent to $e_1$, and the other blue edge is not adjacent to $e_2$, then there are
\begin{equation*}
	f_1(t)=t\binom{t-1}{3}
\end{equation*}
rainbow copies of $K_{1,3}$ in $K_t$; if the other red edge is adjacent to $e_1$, and the other blue edge is not adjacent to $e_2$, then there are
\begin{equation*}
f_2(t)=(t-1)\binom{t-1}{3}+\binom{t-3}{3}+2\binom{t-3}{2}
\end{equation*}
rainbow copies of $K_{1,3}$ in $K_t$; if the other red edge is adjacent to $e_1$, and the other blue edge is adjacent to $e_2$, then there are
\begin{equation*}
f_3(t)=(t-2)\binom{t-1}{3}+2\binom{t-3}{3}+4\binom{t-3}{2}
\end{equation*}
rainbow copies of $K_{1,3}$ in $K_t$.

Next, we compare the sizes of $f_1(t), f_2(t), f_3(t), f_4(t), f_5(t)$ and $f_6(t)$. Based on the practical significance of counting in this paper, we know that the count of rainbow copies of $K_{1,3}$ cannot be negative. For example, for $f_6(t)$, the expression is not applicable when $4\le t\le 6$. In other words, for $4\le t\le 6$, $f_6(t)$ can be written as a piecewise expression. But for the convenience of calculation, we only define in the operations of expressions for $t$ in $f_1(t), f_2(t), f_3(t), f_4(t), f_5(t)$ and $f_6(t)$ that when integers $a<b$, we have $\binom{a}{b}\equiv 0$ and $a-b\equiv 0$.

For $t=4$, we have
\begin{equation*}
	f_1(4)=4, f_2(4)=3, f_3(4)=2, f_4(4)=1, f_5(4)=3, f_6(4)=3.
\end{equation*}	
Thus, $\min\{f_1(4), f_2(4), f_3(4), f_4(4), f_5(4), f_6(4)\}=1.$

For $t=5$, we have
\begin{equation*}
	f_1(5)=20, f_2(5)=18, f_3(5)=16, f_4(5)=14, f_5(5)=16, f_6(5)=16.
\end{equation*}	
Thus, $\min\{f_1(5), f_2(5), f_3(5), f_4(5), f_5(5), f_6(5)\}=14.$

For $t=6$, we have
\begin{equation*}
	f_1(6)=60, f_2(6)=57, f_3(6)=54, f_4(6)=51, f_5(6)=53, f_6(6)=54.
\end{equation*}	
Thus, $\min\{f_1(6), f_2(6), f_3(6), f_4(6), f_5(6), f_6(6)\}=51.$

For $t=7$, we have
\begin{equation*}
	f_1(7)=140, f_2(7)=136, f_3(7)=132, f_4(7)=128, f_5(7)=130, f_6(7)=132.
\end{equation*}	
Thus, $\min\{f_1(7), f_2(7), f_3(7), f_4(7), f_5(7), f_6(7)\}=128.$

For $t\ge 8$ and $1\le i\le 6$, let $f_{ii}(t)=f_i(t)-(t-3)\binom{t-1}{3}-\frac{1}{2}t^3+3t^2-\frac{5}{2}t$, then
\begin{equation*}
	f_{11}(t)=3t-3, f_{22}(t)=2t, f_{33}(t)=t+3, f_{44}(t)=6, f_{55}(t)=8, f_{66}(t)=t+3.
\end{equation*}	
Therefore, for $t\ge 8$,
\begin{equation*}
	\min\{f_{11}(t), f_{22}(t), f_{33}(t), f_{44}(t), f_{55}(t), f_{66}(t)\}=f_{44}(t)=6,
\end{equation*}	
and thus,
\begin{equation*}
\min\{f_1(t), f_2(t), f_3(t), f_4(t), f_5(t), f_6(t)\}=f_4(t)=(t-3)\binom{t-1}{3}+3\binom{t-3}{3}+6\binom{t-3}{2}.
\end{equation*}	
Based on the above discussion, we have
\begin{equation*}
	\min\{f_1(t), f_2(t), f_3(t), f_4(t), f_5(t), f_6(t)\}=\left\{
	\begin{array}{ll}
		(t-3)\binom{t-1}{3}, & t=4;\\
		(t-3)\binom{t-1}{3}+6\binom{t-3}{2}, &  t=5;\\
		(t-3)\binom{t-1}{3}+3\binom{t-3}{3}+6\binom{t-3}{2}, & t\ge 6.\\
	\end{array}\right.
\end{equation*}
The result thus follows.
\end{proof}

\subsection{Rainbow $P_{4}^{+}$}
According to Theorems~\ref{th-Star-Structure} and~\ref{th-P4+-Structure}, we directly give the following observation.

\begin{observation}\label{obv}
For integers $k\ge 5$, if $\operatorname{gr}_k(K_{1,3}:H)\ge 5$, then
$$\operatorname{gr}_k(P_{4}^{+}:H)=\operatorname{gr}_k(K_{1,3}:H).$$
\end{observation}

From Observation~\ref{obv} and Theorem~\ref{k-color-Star-H-general}, the following corollary can be directly deduced.

\begin{corollary}\label{k-color-P4+-H-general}
	Let integer $k\ge 5$. If $H$ is a subgraph of the balanced complete $(k-1)$-partite graph $K_{(k-1)\times 2}$, then
	$$\operatorname{gr}_k(P_4^{+}:H)=
	\begin{cases}
		5, &  5\le k\le 6;\\
		\left\lceil \frac{1+\sqrt{1+8k}}{2} \right\rceil, & k\ge 7.
	\end{cases}$$
\end{corollary}
Noticing that $P_4^{+}$ is obtained by adding a pendent edge to a leaf vertex at $K_{1,3}$, there are $t\binom{t-1}{3}\cdot 3(t-4)$ different $P_4^{+}$ in $K_t$. We can also calculate $|\operatorname{Aut}(P_4^{+})|=2$ from Proposition~\ref{Prop-counting}. Therefore, we directly provide the following corollary.
\begin{corollary}\label{THM-GM-k-color-P4+}
	For integers $k$ and $t$ satisfying $k=\binom{t}{2}\ge 10$, and a subgraph $H$ of the balanced complete $(k-1)$-partite graph $K_{(k-1)\times 2}$ with $|E(H)|\ge 2$, we have
\begin{equation*}	
	\operatorname{GM}_{k}(P_4^{+}:H)=\frac{5!\binom{t}{5}}{|\operatorname{Aut}(P_{4}^{+})|}=60\binom{t}{5}.
\end{equation*}	
\end{corollary}

\begin{theorem}\label{THM-GM-k-1-color-P4+}
	For integers $k$ and $t$ satisfying $k=\binom{t}{2}\ge 10$, and a subgraph $H$ of the balanced complete $(k-2)$-partite graph $K_{(k-2)\times 2}$ with $|E(H)|\ge 3$, we have
\begin{equation*}
	\operatorname{GM}_{k-1}(P_4^{+}:H)=60\binom{t}{5}-5(t-3)(t-4).
\end{equation*}
\end{theorem}
\begin{proof}
It follows from Corollary~\ref{k-color-P4+-H-general} that $\operatorname{gr}_{k-1}(P_4^{+}:H)=t$. Consider any $(k-1)$-edge-coloring of $K_t$. Since $|E(K_t)|=\binom{t}{2}$ and each color is used at least once, it follows that there are only two edges, say $e_1$ and $e_2$, with the same color in $K_t$. Since $|E(H)|\ge 3$, it follows that we do not need to consider the number of monochromatic $H$ in $K_t$. According to Corollary~\ref{THM-GM-k-color-P4+}, there are $60\binom{t}{5}$ different $P_4^{+}$ in $K_t$, and we only need to find the number of different $P_4^{+}$ containing the edges $e_1$ and $e_2$. This is because only $P_4^{+}$ containing edges $e_1$ and $e_2$ are not rainbow, and all other $P_4^{+}$ are rainbow. If $e_1\nsim e_2$, then according to Lemma~\ref{Lem: Conut P4+ in complete graph} that there are $8(t-4)$ different $P_4^{+}$ in $K_t$ that contain edges $e_1$ and $e_2$. If $e_1\sim e_2$, then according to Lemma~\ref{Lem: Conut P4+ in complete graph} that there are $5(t-3)(t-4)$ different $P_4^{+}$ in $K_t$ that contain edges $e_1$ and $e_2$. Noticing that $8(t-4)<5(t-3)(t-4)$ for $t\ge 5$, the result thus follows.
\end{proof}

\begin{theorem}\label{THM-GM-k-2-color-P4+}
	For integers $k$ and $t$ satisfy $k=\binom{t}{2}\ge 10$, if $H$ is a subgraph of the balanced complete $(k-3)$-partite graph $K_{(k-3)\times 2}$ with $|E(H)|\ge 4$, then we have
	\begin{equation*}
		\operatorname{GM}_{k-2}(P_4^{+}:H)=60\binom{t}{5}-15(t-3)(t-4).
	\end{equation*}
\end{theorem}
\begin{proof}
It follows from Corollary~\ref{k-color-P4+-H-general} that $\operatorname{gr}_{k-2}(P_4^{+}:H)=t$. Consider $(k-2)$-edge-coloring of $K_t$. Since $|E(H)|\ge 4$, it follows that we do not need to consider the number of monochromatic $H$ in $K_t$.  Since each color is used at least once, there are only the following two cases. Due to the arbitrariness of colors, we can describe them using specific color names such as red and blue. Next, we calculate the number of different $P_{4}^{+}$ containing two or more edges with the same color. The following counting bases are all based on Lemma~\ref{Lem: Conut P4+ in complete graph}.
\setcounter{case}{0}
\begin{case}
There are three red edges $e_1$, $e_2$ and $e_3$. The remaining edges are not red and the colors of any two remaining edges are not the same.
\end{case}
Assume that the edges $e_1$, $e_2$ and $e_3$ form a red $3P_2$. In this subcase, there are $24(t-4)$ different $P_{4}^{+}$ containing two red edges.

Assume that the edges $e_1$, $e_2$ and $e_3$ form a red $P_3\cup P_2$. In this subcase, there are $5(t-3)(t-4)$ different $P_{4}^{+}$ containing red $P_3$ and $2(8(t-4)-2)$ different $P_{4}^{+}$ without red $P_3$. So there are a total of $(5t+1)(t-4)-4$ different $P_{4}^{+}$ containing two or more red edges.

Assume that the edges $e_1$, $e_2$ and $e_3$ form a red $K_{1,3}$. In this subcase, there are $3(2(t-3)(t-4)+3(t-4)^2)$ different $P_{4}^{+}$ containing two red edges and $3(t-4)$ different $P_{4}^{+}$ containing three red edges. So there are a total of $3(5t-17)(t-4)$ different $P_{4}^{+}$ containing two or more red edges.

Assume that the edges $e_1$, $e_2$ and $e_3$ form a red $K_{3}$. In this subcase, there are $15(t-3)(t-4)$ different $P_{4}^{+}$ containing two red edges.

Assume that the edges $e_1$, $e_2$ and $e_3$ form a red $P_4$. In this subcase, there are $2(5(t-3)(t-4)-2(t-4))+3\cdot 2(t-4)=2(5t-14)(t-4)$ different $P_{4}^{+}$ containing two red edges and $2(t-4)$ different $P_{4}^{+}$ containing three red edges. So there are a total of $2(5t-13)(t-4)$ different $P_{4}^{+}$ containing two or more red edges.

\begin{case}
There are two red edges $e_1, e_2$ and two blue edges $e_3, e_4$. The remaining edges are not red or blue and the colors of any two remaining edges are not the same.
\end{case}
If $e_1\nsim e_2$ and $e_3\nsim e_4$, then there are $16(t-4)$ different $P_{4}^{+}$ containing two edges with the same color; if $e_1\sim e_2$ and $e_3\nsim e_4$, then there are at most $8(t-4)+5(t-3)(t-4)=(5t-7)(t-4)$ different $P_{4}^{+}$ containing two edges with the same color; if $e_1\sim e_2$ and $e_3\sim e_4$, then there are at most $10(t-3)(t-4)$ different $P_{4}^{+}$ containing two edges with the same color.

Let
\begin{equation*}
f_1(t)=24(t-4), f_2(t)=(5t+1)(t-4)-4,
\end{equation*}
\begin{equation*}
f_3(t)=15(t-3)(t-4), f_4(t)=2(5t-13)(t-4), f_5(t)=(5t-7)(t-4).
\end{equation*}
Based on the data calculated from the eight subcases above, we need to compare the sizes of $f_1(t), f_2(t), f_3(t), f_4(t)$ and $f_5(t)$.

For $1\le i\le 5$, let $f_{ii}(t)=\frac{f_i(t)}{t-4}$. Then
\begin{equation*}
	f_{11}(t)=24, f_{22}(t)=5t+1-\frac{4}{t-4}, f_{33}(t)=15(t-3), f_{44}(t)=2(5t-13), f_{55}(t)=5t-7.
\end{equation*}	
Therefore, for $t\ge 5$
\begin{equation*}
	\max\{f_{11}(t), f_{22}(t), f_{33}(t), f_{44}(t), f_{55}(t)\}=f_{33}(t)=15(t-3),
\end{equation*}
and thus,
\begin{equation*}
\max\{f_1(t), f_2(t), f_3(t), f_4(t), f_5(t)\}=f_3(t)=15(t-3)(t-4).
\end{equation*}
The result thus follows.
\end{proof}

\subsection{Rainbow $P_{4}$}
From Theorem~\ref{th-3-path-Structure}, we directly obtain the following corollary.

\begin{corollary}\label{k-color-P4-H-general}
	For a graph $H$ and integer $k\ge 4$, we have
\begin{equation*}
	\operatorname{gr}_k(P_4:H)=\left\lceil \frac{1+\sqrt{1+8k}}{2} \right\rceil.
\end{equation*}
\end{corollary}

According to Corollary~\ref{k-color-P4-H-general} and Proposition~\ref{Prop-counting}, the following corollary can be directly deduced.

\begin{corollary}\label{Cor-GM-k-color-P4}
	For a graph $H$ with $|E(H)|\ge 2$ and integers $k$ and $t$ satisfying $k=\binom{t}{2}\ge 6$, we have
	\begin{equation*}
		\operatorname{GM}_{k}(P_4:H)=\frac{4!\binom{t}{4}}{|\operatorname{Aut}(P_4)|}=12\binom{t}{4}.
	\end{equation*}
\end{corollary}

\begin{theorem}\label{THM-GM-k-1-color-P4}
For a graph $H$ with $|E(H)|\ge 3$ and integers $k$ and $t$ satisfying $k=\binom{t}{2}\ge 6$, we have
\begin{equation*}
\operatorname{GM}_{k-1}(P_4:H)=\left\{
  \begin{array}{ll}
    8, &  t= 4;\\
    12\binom{t}{4}-2(t-3), & t\ge 5.
  \end{array}\right.
\end{equation*}
\end{theorem}
\begin{proof}
It follows from Corollary~\ref{k-color-P4-H-general} that $\operatorname{gr}_{k-1}(P_4:H)=t$. Consider any $(k-1)$-edge-coloring of $K_t$. Since $|E(K_t)|=\binom{t}{2}$ and each color is used at least once, it follows that there are only two edges, say $e_1$ and $e_2$, with the same color in $K_t$. Since $|E(H)|\ge 3$, it follows that we do not need to consider the number of monochromatic $H$ in $K_t$. According to Corollary~\ref{Cor-GM-k-color-P4}, there are $12\binom{t}{4}$ different $P_4$ in $K_t$, and we only need to find the number of different $P_4$ containing the edges $e_1$ and $e_2$. This is because only $P_4$ containing edges $e_1$ and $e_2$ are not rainbow, and all other $P_4$ are rainbow. If $e_1\nsim e_2$, then according to Lemma~\ref{Lem: Conut P4 in complete graph} that there are $4$ different $P_4$ in $K_t$ that contain edges $e_1$ and $e_2$. If $e_1\sim e_2$, then according to Lemma~\ref{Lem: Conut P4 in complete graph} that there are $2(t-3)$ different $P_4$ in $K_t$ that contain edges $e_1$ and $e_2$. Noticing that $4>2(t-3)$ for $t=4$ and $4\le 2(t-3)$ for $t\ge 5$, the result thus follows.
\end{proof}

\begin{theorem}\label{THM-GM-k-2-color-P4}
For a graph $H$ with $|E(H)|\ge 4$ and integers $k$ and $t$ satisfying  $k=\binom{t}{2}\ge 6$, we have
\begin{equation*}
 \operatorname{GM}_{k-2}(P_4:H)=\left\{
	\begin{array}{ll}
		4, &  t= 4;\\
		12\binom{t}{4}-6(t-3), & t\ge 5.
 \end{array}\right.
\end{equation*}
\end{theorem}
\begin{proof}
It follows from Corollary~\ref{k-color-P4-H-general} that $\operatorname{gr}_{k-2}(P_4:H)=t$. Consider $(k-2)$-edge-coloring of $K_t$. Since $|E(H)|\ge 4$, it follows that we do not need to consider the number of monochromatic $H$ in $K_t$. Since each color is used at least once, there are only the following two cases. Due to the arbitrariness of colors, we can describe them using specific color names such as red and blue. Next, we calculate the number of different $P_{4}$ containing two or more edges with the same color. The following counting bases are all based on Lemma~\ref{Lem: Conut P4 in complete graph}.
\setcounter{case}{0}
\begin{case}
There are three red edges $e_1$, $e_2$ and $e_3$. The remaining edges are not red and the colors of any two remaining edges are not the same.
\end{case}

Assume that the edges $e_1$, $e_2$ and $e_3$ form a red $3P_2$. In this subcase, there are $12$ different $P_{4}$ containing two red edges.

Assume that the edges $e_1$, $e_2$ and $e_3$ form a red $P_3\cup P_2$. In this subcase, there are $8+2(t-3)=2(t+1)$ different $P_{4}$ containing two red edges.

Assume that the edges $e_1$, $e_2$ and $e_3$ form a red $K_{1,3}$. In this subcase, there are $6(t-3)$ different $P_{4}$ containing two red edges.

Assume that the edges $e_1$, $e_2$ and $e_3$ form a red $K_{3}$. In this subcase, there are $6(t-3)$ different $P_{4}$ containing two red edges.

Assume that the edges $e_1$, $e_2$ and $e_3$ form a red $P_4$. In this subcase, there are $4+2(t-3)+2(t-4)=2(2t-5)$ different $P_{4}$ containing two or more red edges.

\begin{case}
There are two red edges $e_1, e_2$ and two blue edges $e_3, e_4$. The remaining edges are not red or blue and the colors of any two remaining edges are not the same.
\end{case}
If $e_1\nsim e_2$ and $e_3\nsim e_4$, then there are $8$ different $P_{4}$ containing two edges with the same color; if $e_1\sim e_2$ and $e_3\nsim e_4$, then there are $4+2(t-3)=2(t-1)$ different $P_{4}$ containing two edges with the same color; if $e_1\sim e_2$ and $e_3\sim e_4$, then there are at most $4(t-3)$ different $P_{4}$ containing two edges with the same color.

We first consider the result when $t=4$. Noticing that there are no red $3P_2$ or $P_3\cup P_2$ in a $4$-edge-colored $K_4$. 
Thus there are at most $8$ different $P_4$ in a $4$-edge-colored $K_4$ that contain two or more edges of the same color.

Let $f_1(t)=12, f_2(t)=2(t+1), f_3(t)=6(t-3), f_4(t)=2(2t-5)$.
Based on the data calculated from the eight subcases above, we need to compare the sizes of $f_1(t), f_2(t), f_3(t)$ and $f_4(t)$.

For $t\ge 5$, we have
\begin{equation*}
	\max\{f_1(t), f_2(t), f_3(t), f_4(t)\}=f_3(t)=6(t-3).
\end{equation*}
The result thus follows.
\end{proof}

\subsection{Rainbow $P_{5}$}
In 2023, Zou, Wang, Lai and Mao in~\cite{ZWLM2023} provided results on the Gallai-Ramsey number for rainbow $P_5$.

\begin{theorem}{\upshape \cite{ZWLM2023}}\label{k-color-P5-H-general}
	For a graph $H$ and an integer $k\ge 5$, we have
	$$\operatorname{gr}_k(P_5:H)=
	\begin{cases}
		\max\left\{\left\lceil \frac{1+\sqrt{1+8k}}{2} \right\rceil,5\right\}, & k\ge |V(H)|+1;\\
		|V(H)|+1, &  \text{$k=|V(H)|$ and $H$ is not a complete graph;}\\
		(|V(H)|-1)^2+1, & \text{$k=|V(H)|$ and $H$ is a complete graph.}
	\end{cases}$$
\end{theorem}

According to Theorem~\ref{k-color-P5-H-general} and Proposition~\ref{Prop-counting}, the following corollary can be directly deduced.

\begin{corollary}\label{Cor-GM-k-color-P5}
	For a graph $H$ with $|E(H)|\ge 2$ and integers $k$ and $t$ satisfying $k=\binom{t}{2}\ge \max\{|V(H)|+1,10\}$, we have
	\begin{equation*}
		\operatorname{GM}_{k}(P_5:H)=\frac{5!\binom{t}{5}}{|\operatorname{Aut}(P_5)|}=60\binom{t}{5}.
	\end{equation*}
\end{corollary}

\begin{theorem}\label{THM-GM-k-1-color-P5}
	For a graph $H$ with $|E(H)|\ge 3$ and integers $k$ and $t$ satisfying $k=\binom{t}{2}\ge \max\{|V(H)|+2,10\}$, we have
	\begin{equation*}
		\operatorname{GM}_{k-1}(P_5:H)=\left\{
		\begin{array}{ll}
			60\binom{t}{5}-12(t-4), & 5\le t\le 6;\\
			60\binom{t}{5}-3(t-3)(t-4), & t\ge 7.
		\end{array}\right.
	\end{equation*}
\end{theorem}
\begin{proof}
It follows from Theorem~\ref{k-color-P5-H-general} that $\operatorname{gr}_{k-1}(P_5:H)=t$. Consider any $(k-1)$-edge-coloring of $K_t$. Since $|E(K_t)|=\binom{t}{2}$ and each color is used at least once, it follows that there are only two edges, say $e_1$ and $e_2$, with the same color in $K_t$. Since $|E(H)|\ge 3$, it follows that we do not need to consider the number of monochromatic $H$ in $K_t$. According to Corollary~\ref{Cor-GM-k-color-P5}, there are $60\binom{t}{5}$ different $P_5$ in $K_t$, and we only need to find the number of different $P_5$ containing the edges $e_1$ and $e_2$. This is because only $P_5$ containing edges $e_1$ and $e_2$ are not rainbow, and all other $P_5$ are rainbow. If $e_1\nsim e_2$, then according to Lemma~\ref{Lem: Conut P5 in complete graph} that there are $12(t-4)$ different $P_5$ in $K_t$ that  contain edges $e_1$ and $e_2$. If $e_1\sim e_2$, then according to Lemma~\ref{Lem: Conut P5 in complete graph} that there are $3(t-3)(t-4)$ different $P_5$ in $K_t$ that contain edges $e_1$ and $e_2$. Noticing that $12(t-4)>3(t-3)(t-4)$ for $5 \le t\le 6$ and $12(t-4)\le 3(t-3)(t-4)$ for $t\ge 7$, the result thus follows.
\end{proof}

\begin{theorem}\label{THM-GM-k-2-color-P5}
	For a graph $H$ with $|E(H)|\ge 4$ and integers $k$ and $t$ satisfying $k=\binom{t}{2}\ge \max\{|V(H)|+3,10\}$, we have
	\begin{equation*}
		\operatorname{GM}_{k-2}(P_5:H)=\left\{
		\begin{array}{ll}
			38, & t=5;\\
			288, & t=6;\\
			60\binom{t}{5}-9(t-3)(t-4), & t\ge 7.
		\end{array}\right.
	\end{equation*}
\end{theorem}
\begin{proof}
It follows from Theorem~\ref{k-color-P5-H-general} that $\operatorname{gr}_{k-2}(P_5:H)=t$. Consider $(k-2)$-edge-coloring of $K_t$. Since $|E(H)|\ge 4$, it follows that we do not need to consider the number of monochromatic $H$ in $K_t$. Since each color is used at least once, there are only the following two cases. Due to the arbitrariness of colors, we can describe them using specific color names such as red and blue. Next, we calculate the number of different $P_5$ containing two or more edges with the same color. The following counting bases are all based on Lemma~\ref{Lem: Conut P5 in complete graph}.
\setcounter{case}{0}
\begin{case}
There are three red edges $e_1$, $e_2$ and $e_3$. The remaining edges are not red and the colors of any two remaining edges are not the same.
\end{case}
Assume that the edges $e_1$, $e_2$ and $e_3$ form a red $3P_2$. In this subcase, there are $36(t-4)$ different $P_{5}$ containing two red edges.
		
Assume that the edges $e_1$, $e_2$ and $e_3$ form a red $P_3\cup P_2$. Let $P_3=e_1e_2$ and $P_2=e_3$. In this subcase, there are $3(t-3)(t-4)$ different $P_{5}$ containing red edges $e_1$ and $e_2$, $12(t-4)-4$ different $P_{5}$ only containing red edges $e_1$ and $e_3$, and symmetrically $12(t-4)-4$ different $P_{5}$ only containing red edges $e_2$ and $e_3$. So there are a total of $3(t-3)(t-4)+2(12(t-4)-4)=3(t+5)(t-4)-8$ different $P_{5}$ containing two or more red edges.
		
Assume that the edges $e_1$, $e_2$ and $e_3$ form a red $K_{1,3}$. In this subcase, there are $9(t-3)(t-4)$ different $P_{5}$ containing two red edges.
		
Assume that the edges $e_1$, $e_2$ and $e_3$ form a red $K_{3}$. In this subcase, there are $9(t-3)(t-4)$ different $P_{5}$ containing two red edges.
		
Assume that the edges $e_1$, $e_2$ and $e_3$ form a red $P_4$. In this subcase, there are $12(t-4)$ different $P_{5}$ containing red $2P_2$ and $2(3(t-3)(t-4)-2(t-4))=2(3t-11)(t-4)$ different $P_{5}$ without red $2P_2$. So there are a total of $12(t-4)+2(3t-11)(t-4)=2(3t-5)(t-4)$ different $P_{5}$ containing two or more red edges.

\begin{case}
There are two red edges $e_1, e_2$ and two blue edges $e_3, e_4$. The remaining edges are not red or blue and the colors of any two remaining edges are not the same.
\end{case}
If $e_1\nsim e_2$ and $e_3\nsim e_4$, then there are at most $24(t-4)$ different $P_{5}$ containing two edges with the same color; if $e_1\sim e_2$ and $e_3\nsim e_4$, then there are at most $12(t-4)+3(t-3)(t-4)=3(t+1)(t-4)$ different $P_{5}$ containing two edges with the same color; if $e_1\sim e_2$ and $e_3\sim e_4$, then there are at most $6(t-3)(t-4)$ different $P_{5}$ containing two edges with the same color.
	
Let
\begin{equation*}
	f_1(t)=36(t-4), f_2(t)=3(t+5)(t-4)-8,
\end{equation*}
\begin{equation*}
	f_3(t)=9(t-3)(t-4), f_4(t)=2(3t-5)(t-4), f_5(t)=3(t+1)(t-4).
\end{equation*}
Based on the data calculated from the eight subcases above, we need to compare the sizes of $f_1(t), f_2(t), f_3(t), f_4(t)$ and $f_5(t)$.

For $1\le i\le 5$, let $f_{ii}(t)=\frac{f_i(t)}{t-4}$. Then
\begin{equation*}
	f_{11}(t)=36, f_{22}(t)=3(t+5)-\frac{8}{t-4}, f_{33}(t)=9(t-3), f_{44}(t)=2(3t-5), f_{55}=3(t+1).
\end{equation*}

For $t=5$, note that there are no red $3P_2$ in a $8$-edge-colored $K_5$. Thus
\begin{equation*}
	\max\{f_{22}(t), f_{33}(t), f_{44}(t), f_{55}(t)\}=f_{22}(t)=3(t+5)-\frac{8}{t-4}=22,
\end{equation*}
and thus
\begin{equation*}
	\max\{f_2(t), f_3(t), f_4(t), f_5(t)\}=f_3(t)=3(t+5)(t-4)-8=22.
\end{equation*}
	
For $t=6$, we have
\begin{equation*}
	\max\{f_{11}(t), f_{22}(t), f_{33}(t), f_{44}(t), f_{55}(t)\}=f_{11}(t)=36,
\end{equation*}
and thus
\begin{equation*}
	\max\{f_1(t), f_2(t), f_3(t), f_4(t), f_5(t)\}=f_1(t)=36(t-4)=72.
\end{equation*}	
	
For $t\ge 7$, we have
\begin{equation*}
	\max\{f_{11}(t), f_{22}(t), f_{33}(t), f_{44}(t), f_{55}(t)\}=f_{33}(t)=9(t-3),
\end{equation*}
and thus
\begin{equation*}
	\max\{f_1(t), f_2(t), f_3(t), f_4(t), f_5(t)\}=f_3(t)=9(t-3)(t-4).
\end{equation*}
The result thus follows.
\end{proof}

\section{Results for bipartite Gallai-Ramsey multiplicity}
We consider three kinds of rainbow graphs $P_{4}$,  $P_{5}$ and $K_{1,3}$, respectively, in the following two subsections.

\subsection{Rainbow $P_{4}$}

\begin{theorem}\label{k-color-P4-H-general-bGR}
	Let integer $k\ge 3$. If $H$ is a subgraph of $K_{1,k}$, then
\begin{equation*}
	\operatorname{bgr}_k(P_4:H)=\left\lceil\sqrt{k}\right\rceil.
\end{equation*}
\end{theorem}
\begin{proof}
The lower bound follows from Lemma~\ref{basic-lower-bound-lemma}. For the upper bound, we consider an arbitrary $k$-edge-coloring of $K_{N,N} \ (N\ge \left\lceil\sqrt{k}\right\rceil)$. Let $(U,V)$ be the bipartition of $K_{N,N}$ and suppose to the contrary that $K_{N,N}$ contains neither a rainbow subgraph $P_4$ nor a monochromatic subgraph $H$. Noticing that $\left\lceil\sqrt{k}\right\rceil\le k-1$ for $k\geq 3$. If $\left\lceil\sqrt{k}\right\rceil\le N\le k-1$, then it follows from Proposition~\ref{k-color-have-rainbowP4-bipartite} that there is always a rainbow $P_4$, and the result thus follows. Next we assume $N\ge k$. It follows from Theorem~\ref{th-P4-bGR-Structure} that the Colored Structure 3 occurs. Thus $U$ can be partitioned into $k$ non-empty parts $U_1, U_2,\ldots,U_k$ such that all the edges between $U_i$ and $V$ have color $i$, $i\in\{1,2,\ldots,k\}$. Since $H$ is a subgraph of $K_{1,k}$ and $|V|=N \ge k$, it follows that there is a monochromatic $H$, a contradiction. The result thus follows.
\end{proof}
It is easy to calculate that when $t\ge 2$, there are $t^2(t-1)^2$ different $P_4$ in $K_{t,t}$. Therefore, the following corollary can be directly derived from Theorem~\ref{k-color-P4-H-general-bGR}.
\begin{corollary}\label{THM-biGM-k-color-P4}
	For integers $k$ and $t$ satisfying $k=t^2\ge 4$, and a subgraph $H$ of $K_{1,k}$ with $|E(H)|\ge 2$, we have
\begin{equation*}
	\operatorname{bi-GM}_{k}(P_4:H)=t^2(t-1)^2.
\end{equation*}
\end{corollary}

\begin{theorem}\label{THM-biGM-k-1-color-P4}
	For integers $k$ and $t$ satisfying $k=t^2\ge 4$,  and a subgraph $H$ of $K_{1,k-1}$ with $|E(H)|\ge 3$, we have
	\begin{equation*}
		\operatorname{bi-GM}_{k-1}(P_4:H)=t^2(t-1)^2-2(t-1).
	\end{equation*}
\end{theorem}
\begin{proof}
It follows from Theorem~\ref{k-color-P4-H-general-bGR} that $\operatorname{bgr}_{k-1}(P_4:H)=t$. Consider any $(k-1)$-edge-coloring of $K_{t,t}$. Since $|E(K_{t,t})|=t^2$ and each color is used at least once, it follows that there are only two edges, say $e_1$ and $e_2$, with the same color in $K_{t,t}$. Since $|E(H)|\ge 3$, it follows that we do not need to consider the number of monochromatic $H$ in $K_{t,t}$. According to Corollary~\ref{THM-biGM-k-color-P4}, there are $t^2(t-1)^2$ different $P_4$ in $K_{t,t}$, and we only need to find the number of different $P_4$ containing the edges $e_1$ and $e_2$. This is because only $P_4$ containing edges $e_1$ and $e_2$ are not rainbow, and all other $P_4$ are rainbow. If $e_1\nsim e_2$, then according to Lemma~\ref{Lem: Conut P4 in complete bipartite graph} that there are $2$ different $P_4$ in $K_{t,t}$ that contain edges $e_1$ and $e_2$. If $e_1\sim e_2$, then according to Lemma~\ref{Lem: Conut P4 in complete bipartite graph} that there are $2(t-1)$ different $P_4$ in $K_{t,t}$ that contain edges $e_1$ and $e_2$. Noticing that $2\le 2(t-1)$ for $t\ge 2$, the result thus follows.
\end{proof}

\begin{theorem}\label{THM-biGM-k-2-color-P4}
	For integers $k$ and $t$ satisfying $k=t^2\ge 9$, and a subgraph $H$ of $K_{1,k-2}$ with $|E(H)|\ge 4$, we have
	\begin{equation*}
		\operatorname{bi-GM}_{k-2}(P_4:H)=t^2(t-1)^2-6(t-1).
	\end{equation*}
\end{theorem}
\begin{proof}
It follows from Theorem~\ref{k-color-P4-H-general-bGR} that $\operatorname{bgr}_{k-2}(P_4:H)=t$. Consider $(k-2)$-edge-coloring of $K_{t,t}$. Since $|E(H)|\ge 4$, it follows that we do not need to consider the number of monochromatic $H$ in $K_{t,t}$. Since each color is used at least once, there are only the following two cases. Due to the arbitrariness of colors, we can describe them using specific color names such as red and blue. Next, we calculate the number of different $P_{4}$ containing two or more edges with the same color. The following counting bases are all based on Lemma~\ref{Lem: Conut P4 in complete bipartite graph}.

\setcounter{case}{0}
\begin{case}
There are three red edges $e_1$, $e_2$ and $e_3$. The remaining edges are not red and the colors of any two remaining edges are not the same.
\end{case}
Assume that the edges $e_1$, $e_2$ and $e_3$ form a red $3P_2$. In this subcase, there are $6$ different $P_{4}$ containing two red edges.
		
Assume that the edges $e_1$, $e_2$ and $e_3$ form a red $P_3\cup P_2$. In this subcase, there are $4+2(t-1)=2(t+1)$ different $P_{4}$ containing two red edges.
		
Assume that the edges $e_1$, $e_2$ and $e_3$ form a red $K_{1,3}$. In this subcase, there are $6(t-1)$ different $P_{4}$ containing two red edges.
		
Assume that the edges $e_1$, $e_2$ and $e_3$ form a red $P_4$. In this subcase, there are $2+2(t-1)+2(t-2)=4(t-1)$ different $P_{4}$ containing two or more red edges.

\begin{case}
There are two red edges $e_1, e_2$ and two blue edges $e_3, e_4$. The remaining edges are not red or blue and the colors of any two remaining edges are not the same.
\end{case}
If $e_1\nsim e_2$ and $e_3\nsim e_4$, then there are $4$ different $P_{4}$ containing two edges with the same color; if $e_1\sim e_2$ and $e_3\nsim e_4$, then there are $2+2(t-1)=2t$ different $P_{4}$ containing two edges with the same color; if $e_1\sim e_2$ and $e_3\sim e_4$, then there are $4(t-1)$ different $P_{4}$ containing two edges with the same color.

Let $f_1(t)=6, f_2(t)=2(t+1), f_3(t)=6(t-1)$. Based on the data calculated from the seven subcases above, we need to compare the sizes of $f_1(t), f_2(t)$ and $f_3(t)$.

For $t\ge 3$, we have
\begin{equation*}
	\max\{f_1(t), f_2(t), f_3(t)\}=f_3(t)=6(t-1).
\end{equation*}
The result thus follows.
\end{proof}

\subsection{Rainbow $P_{5}$ and $K_{1,3}$}

\begin{theorem}\label{k-color-P5&K13-H-general-bGR}
	Let integer $k\ge 5$. If $H$ is a subgraph of $K_{1,\left\lceil\frac{k-1}{2}\right\rceil}$, then
\begin{equation*}
	\operatorname{bgr}_k(P_5:H)=\operatorname{bgr}_k(K_{1,3}:H)=\left\lceil\sqrt{k}\right\rceil.
\end{equation*}
\end{theorem}
\begin{proof}
The lower bound follows from Lemma~\ref{basic-lower-bound-lemma}. For the upper bound, we consider an arbitrary $k$-edge-coloring of $K_{N,N} \ (N\ge \left\lceil\sqrt{k}\right\rceil)$. Let $(U,V)$ be the bipartition of $K_{N,N}$ and suppose to the contrary that $K_{N,N}$ contains neither a rainbow subgraph $P_5$ nor a monochromatic subgraph $H$. Noticing that $\left\lceil\sqrt{k}\right\rceil\le k-2$ for $k\geq 5$. If $\left\lceil\sqrt{k}\right\rceil\le N\le k-2$, then it follows from Proposition~\ref{k-color-have-rainbowP5-bipartite} that there is always a rainbow $P_5$, and the result thus follows. Next we assume $N\ge k-1$. It follows from Theorems~\ref{th-P5-bGR-Structure} and~\ref{th-K13-bGR-Structure} that either the Colored Structure 4 or Colored Structure 5 occurs. If Colored Structure 4 occurs, then $U$ can be partitioned into two parts $U_1$ and $U_2$ with $|U_1|\ge 1, |U_2|\ge 0$, and $V$ can be partitioned into $k$ parts $V_1,V_2,\ldots,V_k$ with $|V_1|\ge 0$ and $|V_j|\ge 1$, $j\in\{2,3,\ldots,k\}$. Since $N\ge k-1$, it follows from pigeonhole principle that $|U_1|\ge \left\lceil\frac{k-1}{2}\right\rceil$ or $|U_2|\ge \left\lceil\frac{k-1}{2}\right\rceil$. Without loss of generality, we assume that $|U_1|\ge \left\lceil\frac{k-1}{2}\right\rceil$. Noticing that $|V_2|\ge 1$ and all the edges between $V_2$ and $U_1$ have color $2$, there is a monochromatic $H$ with color $2$, a contradiction. If the Colored Structure 5 occurs, then $U$ can be partitioned into $k$ parts $U_1,U_2,\ldots,U_k$ with $|U_1|\ge 0, |U_j|\ge 1$ and $V$ can be partitioned into $k$ parts $V_1,V_2,\ldots,V_k$ with $|V_1|\ge 0, |V_j|\ge 1$, $j\in\{2,3,\ldots,k\}$. Noticing that $ \left\lceil\frac{k-1}{2}\right\rceil< k-2$ for $k\ge 5$, $|V_2|\ge 1$ and all the edges between $V_2$ and $U_3\cup U_4\cup\ldots\cup U_k$ have color $1$, there is a monochromatic $H$ with color $1$, a contradiction. The result thus follows.
\end{proof}

It is easy to calculate that when $t\ge 3$, there are $t^2(t-1)^2(t-2)$ different $P_5$ and $2t\binom{t}{3}$ different $K_{1,3}$ in $K_{t,t}$. Therefore, the following corollary can be directly derived from Theorem~\ref{k-color-P5&K13-H-general-bGR}.
\begin{corollary}\label{THM-biGM-k-color-P5&K13}
	For integers $k$ and $t$ satisfying $k=t^2\ge 9$, and a subgraph $H$ of $K_{1,\left\lceil\frac{k-1}{2}\right\rceil}$ with $|E(H)|\ge 2$, we have
\begin{equation*}
  \operatorname{bi-GM}_{k}(G:H)=\left\{
  \begin{array}{ll}
  	t^2(t-1)^2(t-2), & G=P_5;\\
  	2t\binom{t}{3}, & G=K_{1,3}.
  \end{array}\right.
\end{equation*}
\end{corollary}

\begin{theorem}\label{THM-biGM-k-1-color-P5}
	For integers $k$ and $t$ satisfying $k=t^2\ge 9$, and a subgraph $H$ of $K_{1,\left\lceil\frac{k-2}{2}\right\rceil}$ with $|E(H)|\ge 3$, we have
	\begin{equation*}
		\operatorname{bi-GM}_{k-1}(P_5:H)=t^2(t-1)^2(t-2)-3(t-1)(t-2).
	\end{equation*}
\end{theorem}
\begin{proof}
It follows from Theorem~\ref{k-color-P5&K13-H-general-bGR} that $\operatorname{bgr}_{k-1}(P_5:H)=t$. Consider any $(k-1)$-edge-coloring of $K_{t,t}$. Since $|E(K_{t,t})|=t^2$ and each color is used at least once, it follows that there are only two edges, say $e_1$ and $e_2$, with the same color in $K_{t,t}$. Since $|E(H)|\ge 3$, it follows that we do not need to consider the number of monochromatic $H$ in $K_{t,t}$. According to Corollary~\ref{THM-biGM-k-color-P5&K13}, there are $t^2(t-1)^2(t-2)$ different $P_5$ in $K_{t,t}$, and we only need to find the number of different $P_5$ containing the edges $e_1$ and $e_2$. This is because only $P_5$ containing edges $e_1$ and $e_2$ are not rainbow, and all other $P_5$ are rainbow. If $e_1\nsim e_2$, then according to Lemma~\ref{Lem: Conut P5 in complete bipartite graph} that there are $6(t-2)$ different $P_5$ in $K_{t,t}$ that contain edges $e_1$ and $e_2$. If $e_1\sim e_2$, then according to Lemma~\ref{Lem: Conut P5 in complete bipartite graph} that there are $3(t-1)(t-2)$ different $P_5$ in $K_{t,t}$ that contain edges $e_1$ and $e_2$. Noticing that $6(t-2)\le 3(t-1)(t-2)$ for $t\ge 3$, the result thus follows.
\end{proof}

\begin{theorem}\label{THM-biGM-k-2-color-P5}
	For integers $k$ and $t$ satisfying $k=t^2\ge 9$, and a subgraph $H$ of $K_{1,\left\lceil\frac{k-3}{2}\right\rceil}$ with $|E(H)|\ge 4$, we have
	\begin{equation*}
		\operatorname{bi-GM}_{k-2}(P_5:H)=t^2(t-1)^2(t-2)-9(t-1)(t-2).
	\end{equation*}
\end{theorem}
\begin{proof}
It follows from Theorem~\ref{k-color-P5&K13-H-general-bGR} that $\operatorname{bgr}_{k-2}(P_5:H)=t$. Consider $(k-2)$-edge-coloring of $K_{t,t}$. Since $|E(H)|\ge 4$, it follows that we do not need to consider the number of monochromatic $H$ in $K_{t,t}$. Since each color is used at least once, there are only the following two cases. Due to the arbitrariness of colors, we can describe them using specific color names such as red and blue. Next, we calculate the number of different $P_5$ containing two or more edges with the same color. The following counting bases are all based on Lemma~\ref{Lem: Conut P5 in complete bipartite graph}.

\setcounter{case}{0}
\begin{case}
There are three red edges $e_1$, $e_2$ and $e_3$. The remaining edges are not red and the colors of any two remaining edges are not the same.
\end{case}
Assume that the edges $e_1$, $e_2$ and $e_3$ form a red $3P_2$. In this subcase, there are $18(t-2)$ different $P_{5}$ containing two red edges.
		
Assume that the edges $e_1$, $e_2$ and $e_3$ form a red $P_3\cup P_2$. Let $P_3=e_1e_2$ and $P_2=e_3$. In this subcase, there are $3(t-1)(t-2)$ different $P_{5}$ containing red edges $e_1$ and $e_2$, $6(t-2)-2$ different $P_{5}$ only containing red edges $e_1$ and $e_3$, and symmetrically $6(t-2)-2$ different $P_{5}$ only containing red edges $e_2$ and $e_3$. So there are a total of $3(t-1)(t-2)+2(6(t-2)-2)=3(t+3)(t-2)-4$ different $P_{5}$ containing two or more red edges.
		
Assume that the edges $e_1$, $e_2$ and $e_3$ form a red $K_{1,3}$. In this subcase, there are $9(t-1)(t-2)$ different $P_{5}$ containing two red edges.
		
Assume that the edges $e_1$, $e_2$ and $e_3$ form a red $P_4$.	In this subcase, there are $6(t-2)$ different $P_{5}$ containing red $2P_2$ and $2(3(t-1)(t-2)-2(t-2))=2(3t-5)(t-2)$ different $P_{5}$ without red $2P_2$. So there are a total of $6(t-2)+2(3t-5)(t-2)=2(3t-2)(t-2)$ different $P_{5}$ containing two or more red edges.

\begin{case}
There are two red edges $e_1, e_2$ and two blue edges $e_3, e_4$. The remaining edges are not red or blue and the colors of any two remaining edges are not the same.
\end{case}
If $e_1\nsim e_2$ and $e_3\nsim e_4$, then there are at most $12(t-2)$ different $P_{5}$ containing two edges with the same color; if $e_1\sim e_2$ and $e_3\nsim e_4$, then there are at most $6(t-2)+3(t-1)(t-2)=3(t+1)(t-2)$ different $P_{5}$ containing two edges with the same color; if $e_1\sim e_2$ and $e_3\sim e_4$, then there are at most $6(t-1)(t-2)$ different $P_{5}$ containing two edges with the same color.

Let
\begin{equation*}
	f_1(t)=18(t-2), f_2(t)=3(t+3)(t-2)-4,
\end{equation*}
\begin{equation*}
	f_3(t)=9(t-1)(t-2), f_4(t)=2(3t-2)(t-2), f_5(t)=3(t+1)(t-2).
\end{equation*}
Based on the data calculated from the seven subcases above, we need to compare the sizes of $f_1(t), f_2(t), f_3(t), f_4(t)$ and $f_5(t)$.

For $1\le i\le 5$, let $f_{ii}(t)=\frac{f_i(t)}{t-2}$. Then
\begin{equation*}
	f_{11}(t)=18, f_{22}(t)=3(t+3)-\frac{4}{t-2}, f_{33}(t)=9(t-1), f_{44}(t)=2(3t-2), f_{55}=3(t+1).
\end{equation*}

For $t\ge 3$, we have
\begin{equation*}
	\max\{f_{11}(t), f_{22}(t), f_{33}(t), f_{44}(t), f_{55}(t)\}=f_{33}(t)=9(t-1),
\end{equation*}
and thus
\begin{equation*}
	\max\{f_1(t), f_2(t), f_3(t), f_4(t), f_5(t)\}=f_3(t)=9(t-1)(t-2).
\end{equation*}
The result thus follows.
\end{proof}

\begin{theorem}\label{THM-biGM-k-1-color-K13}
	For integers $k$ and $t$ satisfying $k=t^2\ge 9$, and a subgraph $H$ of $K_{1,\left\lceil\frac{k-2}{2}\right\rceil}$ with $|E(H)|\ge 3$, we have
\begin{equation*}
 \operatorname{bi-GM}_{k-1}(K_{1,3}:H)=\left\{
\begin{array}{ll}
	5, & t=3;\\
	30, & t=4;\\
	(2t-1)\binom{t}{3}+\binom{t-2}{3}+2\binom{t-2}{2}, & t\ge 5.
\end{array}\right.
\end{equation*}
\end{theorem}
\begin{proof}
It follows from Theorem~\ref{k-color-P5&K13-H-general-bGR} that $\operatorname{bgr}_{k-1}(K_{1,3}:H)=t$. Consider any $(k-1)$-edge-coloring of $K_{t,t}$. Since $|E(K_{t,t})|=t^2$ and each color is used at least once, it follows that there are only two edges, say $e_1$ and $e_2$, with the same color in $K_{t,t}$. Since $|E(H)|\ge 3$, it follows that we do not need to consider the number of monochromatic $H$ in $K_{t,t}$. Next, we calculate the number of rainbow copies of $K_{1,3}$ in $K_{t,t}$. If $e_1\nsim e_2$, then this case is equivalent to Corollary~\ref{THM-biGM-k-color-P5&K13}. Hence there are $2t\binom{t}{3}$ rainbow copies of $K_{1,3}$ in $K_{t,t}$. If $e_1\sim e_2$, then $e_1$ and $e_2$ form a monochromatic $P_3$. Without loss of generality, we assume that the edges $e_1$ and $e_2$ are red and vertex $v$ is incident with the edges $e_1$ and $e_2$. We first investigate the number of rainbow copies of $K_{1,3}$ with center $v$ for $t\ge 5$. Noticing that $\deg(v)=t$, the number of rainbow copies of $K_{1,3}$ with center $v$ and without red edges is $\binom{t-2}{3}$, and number of rainbow copies of $K_{1,3}$ with center $v$ and with a red edge is $2\binom{t-2}{2}$. In $K_{t,t}$, there are $(2t-1)\binom{t}{3}$ rainbow copies of $K_{1,3}$ with center in $V(K_{t,t})\setminus \{v\}$. Therefore, the total number of rainbow copies of $K_{1,3}$ in $K_{t,t}$ is $(2t-1)\binom{t}{3}+\binom{t-2}{3}+2\binom{t-2}{2}$. It is easy to verify that when $t\ge 5$, 
\begin{equation*}
\min\left\{2t\binom{t}{3},(2t-1)\binom{t}{3}+\binom{t-2}{3}+2\binom{t-2}{2}\right\}=(2t-1)\binom{t}{3}+\binom{t-2}{3}+2\binom{t-2}{2}.
\end{equation*} 
	
When $t=3$, the number of rainbow copies of $K_{1,3}$ with center $v$ and without red edges is $0$, and number of rainbow copies of $K_{1,3}$ with center $v$ and with a red edge is $0$. In $K_{3,3}$, there are $5\binom{3}{3}=5$ rainbow copies of $K_{1,3}$ with center in $V(K_{3,3})\setminus \{v\}$, therefore $\operatorname{bi-GM}_{8}(K_{1,3}:H)\le 5$. Since $5<6\binom{5}{3}$, it follows that $\operatorname{bi-GM}_{8}(K_{1,3}:H)=5$.

When $t=4$, the number of rainbow copies of $K_{1,3}$ with center $v$ and without red edges is $0$, and number of rainbow copies of $K_{1,3}$ with center $v$ and with a red edge is $2$. In $K_{4,4}$, there are $7\binom{4}{3}=28$ rainbow copies of $K_{1,3}$ with center in $V(K_{4,4})\setminus \{v\}$, and therefore $\operatorname{bi-GM}_{15}(K_{1,3}:H)\le 30$. Since $30<8\binom{4}{3}$, it follows that $\operatorname{bi-GM}_{15}(K_{1,3}:H)=30$.
\end{proof}

\begin{theorem}\label{THM-biGM-k-2-color-K13}
	For integers $k$ and $t$ satisfying $k=t^2\ge 9$, and a subgraph $H$ of $K_{1,\left\lceil\frac{k-3}{2}\right\rceil}$ with $|E(H)|\ge 4$, we have
	\begin{equation*}
		\operatorname{bi-GM}_{k-2}(K_{1,3}:H)=\left\{
		\begin{array}{ll}
			4, & t=3;\\
			28, & t=4;\\
			93, & t=5;\\
			(2t-1)\binom{t}{3}+\binom{t-3}{3}+3\binom{t-3}{2}, & t\ge 6.\\
		\end{array}\right.
	\end{equation*}
\end{theorem}
\begin{proof}
It follows from Theorem~\ref{k-color-P5&K13-H-general-bGR} that $\operatorname{bgr}_{k-2}(K_{1,3}:H)=t$. Consider $(k-2)$-edge-coloring of $K_{t,t}$. Since $|E(H)|\ge 4$, it follows that we do not need to consider the number of monochromatic $H$ in $K_{t,t}$. Noticing that each color needs to be used at least once. We first color any $k-2$ edges in $K_{t,t}$ with $k-2$ colors, and the remaining two edges are temporarily not colored, denoted as $e_1$ and $e_2$. Next, we discuss the edges $e_1$ and $e_2$ in two cases.
\setcounter{case}{0}
\begin{case}
	The edges $e_1$ and $e_2$ have the same color.
\end{case}	
Without loss of generality, we assume that these two edges are red. According to the structure of $K_{t,t}$, it is easy to calculate that if the red edges form a $3P_2$, then there are
\begin{equation*}
	f_1(t)=2t\binom{t}{3}
\end{equation*}
rainbow copies of $K_{1,3}$ in $K_{t,t}$; if the red edges form a $P_3\cup P_2$, then there are
\begin{equation*}
	f_2(t)=(2t-1)\binom{t}{3}+\binom{t-2}{3}+2\binom{t-2}{2}
\end{equation*}
rainbow copies of $K_{1,3}$ in $K_{t,t}$; if the red edges form a $P_4$, then there are
\begin{equation*}
	f_3(t)=(2t-2)\binom{t}{3}+2\binom{t-2}{3}+4\binom{t-2}{2}
\end{equation*}
rainbow copies of $K_{1,3}$ in $K_{t,t}$; if the red edges form a $K_{1,3}$, then there are
\begin{equation*}
	f_4(t)=(2t-1)\binom{t}{3}+\binom{t-3}{3}+3\binom{t-3}{2}
\end{equation*}
rainbow copies of $K_{1,3}$ in $K_{t,t}$.
	
\begin{case}
	The edges $e_1$ and $e_2$ have different colors.
\end{case}
When the edges $e_1$ and $e_2$ form a $P_3$ in $K_{t,t}$, without loss of generality, we assume that $e_1$ is red and $e_2$ is blue. Let $V(P_3)=\{u,v,w\}$ and vertex $v$ be incident with the edges $e_1$ and $e_2$. According to the structure of $K_{t,t}$, it is easy to calculate that if the other red edge is not incident with vertex $u$ or $v$, and the other blue edge is not incident with vertex $v$ or $w$, then there are
\begin{equation*}
	f_1(t)=2t\binom{t}{3}
\end{equation*}
rainbow copies of $K_{1,3}$ in $K_{t,t}$; if the other red edge is incident with vertex $u$ or $v$, and the other blue edge is not incident with vertex $v$ or $w$, then there are
\begin{equation*}
	f_2(t)=(2t-1)\binom{t}{3}+\binom{t-2}{3}+2\binom{t-2}{2}
\end{equation*}
rainbow copies of $K_{1,3}$ in $K_{t,t}$; if the other red edge is incident with vertex $u$, and the other blue edge is incident with vertex $w$, then there are
\begin{equation*}
	f_3(t)=(2t-2)\binom{t}{3}+2\binom{t-2}{3}+4\binom{t-2}{2}
\end{equation*}
rainbow copies of $K_{1,3}$ in $K_{t,t}$; if the other red edge is incident with vertex $v$, and the other blue edge is also incident with vertex $v$, then there are
\begin{equation*}
	f_5(t)=(2t-1)\binom{t}{3}+\binom{t-4}{3}+4\binom{t-4}{2}+4(t-4)
\end{equation*}
rainbow copies of $K_{1,3}$ in $K_{t,t}$.
	
When the edges $e_1$ and $e_2$ form a $2P_2$ in $K_{t,t}$, without loss of generality, we assume that $e_1$ is red and $e_2$ is blue. According to the structure of $K_{t,t}$, it is easy to calculate that if the other red edge is not adjacent to $e_1$, and the other blue edge is not adjacent to $e_2$, then there are
\begin{equation*}
	f_1(t)=2t\binom{t}{3}
\end{equation*}
rainbow copies of $K_{1,3}$ in $K_{t,t}$; if the other red edge is adjacent to $e_1$, and the other blue edge is not adjacent to $e_2$, then there are
\begin{equation*}
	f_2(t)=(2t-1)\binom{t}{3}+\binom{t-2}{3}+2\binom{t-2}{2}
\end{equation*}
rainbow copies of $K_{1,3}$ in $K_{t,t}$; if the other red edge is adjacent to $e_1$, and the other blue edge is adjacent to $e_2$, then there are
\begin{equation*}
	f_3(t)=(2t-2)\binom{t}{3}+2\binom{t-2}{3}+4\binom{t-2}{2}
\end{equation*}
rainbow copies of $K_{1,3}$ in $K_{t,t}$.
	
Next, we compare the sizes of $f_1(t), f_2(t), f_3(t), f_4(t)$ and $f_5(t)$. Based on the practical significance of counting in this paper, we only define in the operations of expressions for $t$ in $f_1(t), f_2(t), f_3(t), f_4(t)$ and $f_5(t)$ that when integers $a<b$, we have $\binom{a}{b}\equiv 0$ and $a-b\equiv 0$.
	
For $t=3$, we have
	\begin{equation*}
		f_1(3)=6, f_2(3)=5, f_3(3)=4, f_4(3)=5, f_5(3)=5.
	\end{equation*}	
	Thus, $\min\{f_1(3), f_2(3), f_3(3), f_4(3), f_5(3)\}=4.$
	
	For $t=4$, we have
	\begin{equation*}
		f_1(4)=32, f_2(4)=30, f_3(4)=28, f_4(4)=28, f_5(4)=28.
	\end{equation*}	
	Thus, $\min\{f_1(4), f_2(4), f_3(4), f_4(4), f_5(4)\}=28.$
	
	For $t=5$, we have
	\begin{equation*}
		f_1(5)=100, f_2(5)=97, f_3(5)=94, f_4(5)=93, f_5(5)=94.
	\end{equation*}	
	Thus, $\min\{f_1(5), f_2(5), f_3(5), f_4(5), f_5(5)\}=93.$
	
	For $t=6$, we have
	\begin{equation*}
		f_1(6)=240, f_2(6)=236, f_3(6)=232, f_4(6)=230, f_5(6)=232.
	\end{equation*}	
	Thus, $\min\{f_1(6), f_2(6), f_3(6), f_4(6), f_5(6)\}=230.$
	
	For $t\ge 7$ and $1\le i\le 5$, let $f_{ii}(t)=f_i(t)-(2t-2)\binom{t}{3}-\frac{1}{3}t^3+t^2+\frac{1}{3}t$, then
	\begin{equation*}
		f_{11}(t)=t, f_{22}(t)=2, f_{33}(t)=-t+4, f_{44}(t)=-2t+8, f_{55}(t)=-t+4.
	\end{equation*}	
Therefore, when $t\ge 7$ we have 
\begin{equation*}
\min\{f_{11}(t), f_{22}(t), f_{33}(t), f_{44}(t), f_{55}(t)\}=f_{44}(t)=-2t+8,
\end{equation*}	
and thus
\begin{equation*}
\min\{f_1(t), f_2(t), f_3(t), f_4(t), f_5(t)\}=f_4(t)=(2t-1)\binom{t}{3}+\binom{t-3}{3}+3\binom{t-3}{2}.
\end{equation*}	
Based on the above discussion, we have
\begin{equation*}
\min\{f_1(t), f_2(t), f_3(t), f_4(t), f_5(t)\}=\left\{
\begin{array}{ll}
 (2t-2)\binom{t}{3}, & t=3;\\
 (2t-1)\binom{t}{3}, &  t=4;\\
 (2t-1)\binom{t}{3}+3\binom{t-3}{2}, &  t=5;\\
 (2t-1)\binom{t}{3}+\binom{t-3}{3}+3\binom{t-3}{2}, & t\ge 6.\\
\end{array}\right.
\end{equation*}
The result thus follows.
\end{proof}

\section*{Acknowledgements}
	The authors would like to thank the anonymous referees very much for their careful reading and helpful comments and suggestions, which improved the clarity of this work. 

\section*{Conflict of interest}
The authors declare that they have no conflict of interest.

\end{document}